\documentclass[final,3p,times]{elsarticle}
\usepackage{amsfonts}
\usepackage{dsfont}
\usepackage{latexsym,amssymb}
\usepackage{amsmath, amsbsy}
\usepackage{amsopn, amstext}
\usepackage{graphicx, xcolor, epstopdf}
\usepackage{caption}
\usepackage{subcaption}
\usepackage{hyperref}
\usepackage{float}
\usepackage{enumitem}
\usepackage{graphicx}
\usepackage{diagbox}
\usepackage{amsmath,bm}
\usepackage{booktabs}
\usepackage{multirow}
\allowdisplaybreaks[4]

\usepackage[marginal]{footmisc}
\newtheorem{lemma}{Lemma}
\newtheorem{theorem}{Theorem}
\newtheorem{remark}{Remark}

\newtheorem{proposition}{Proposition}
\newtheorem{assumption}{Assumption}[section]

    \numberwithin{Fig}{section}

\makeatletter
\renewcommand{\section}{\@startsection{section}{1}{0mm}
  {-0.5\baselineskip}{0.3\baselineskip}{\bf\leftline}}

\let\oldthebibliography\thebibliography
\renewcommand\thebibliography[1]{
  \oldthebibliography{#1}
  \setlength{\itemsep}{0em}
  \setlength{\parskip}{0pt}
}
\numberwithin{equation}{section}
\numberwithin{Lem}{section}
\numberwithin{Rem}{section}
\numberwithin{theorem}{section}

\begin{document}
\begin{frontmatter}

\title{A high order stabilization-free virtual element method for general second-order elliptic eigenvalue problem}

\author[1]{Liangkun Xu}\ead{xlkun@gznu.edu.cn}

\author[1]{Shixi Wang}\ead{wangshixi@gznu.edu.cn}

\author[1]{Yidu Yang}\ead{ydyang@gznu.edu.cn}

\author[1]{Hai Bi\corref{cor1}}\ead{bihaimath@gznu.edu.cn}

\cortext[cor1]{Corresponding Author}

\begin{abstract}
In this paper, we discuss a novel higher-order stabilization-free virtual element method for general second-order elliptic eigenvalue problems. Optimal a priori error estimates are derived for both the approximate eigenspace and eigenvalues. Numerical experiments are conducted on regular convex polygonal meshes, convex-concave polygonal meshes, and concave polygonal meshes. The numerical results validate the effectiveness of the proposed method.
\end{abstract}

\end{frontmatter}

\section{Introduction}
\indent The virtual element method was originally proposed in a conforming formulation by \cite{BeiraoBC2013,BeiraoBM2014}
for Poisson equation. Research indicates that the virtual element method not only retains the advantages of finite element formulations, but is also well-suited for handling complex element geometries. It proves particularly effective in addressing higher-order continuity requirements and allows for the straightforward construction of $C^1$-approximations. Since then, various variants have been developed, including the mixed virtual element method \cite{BrezziFM2014}, nonconforming virtual element method \cite{DeDiosLM2016}, serendipity virtual element method \cite{BeiraoBreMar2016}, h-version and p-version virtual element methods \cite{MascottoPP2018}, extended virtual element method \cite{BenvenutiCM2019}, immersed virtual element method \cite{CaoCG2022}, interior penalty virtual element method \cite{ZhaoZ2023}, and two-grid virtual element method \cite{ChenWZ2023}. The standard virtual element discretization consists of two parts: a consistent term (polynomial) and a stabilization term (non-polynomial). In recent years, the role of the stabilization term in virtual element methods has garnered continuous attention \cite{BoffiGG2020,BerroneBM2022,BeiraoCNRH2023,Mascotto2023,AlzabenBD2025,BerroneBOFA2026}. Notably, in \cite{BerroneBOMA2023} it is pointed out that the absence of the stabilization term can reduce errors and facilitate convergence in strongly anisotropic problems. Furthermore, \cite{BerroneBF2025} mentioned that in classical a posteriori error analysis, the stabilization term appearing on the right-hand side poses certain challenges to adaptive theory. \\

\indent Research on low-order stabilization-free schemes (i.e., methods that do not use stabilization terms) or self-stabilizing virtual element methods began with the early publicly available work in the scientific community by \cite{BerroneBM2025} and \cite{DAltriAMdp2021}.
 Since then, stabilization-free virtual element methods have become a research focus, as evidenced by studies such as \cite{ChenSukumar2023, BertrandCG2023, XuPeng2023, LampertiCP2023, BouchezGB2024, BorioLM2024, BerroneBM2024}. It is important to note that research on stabilization-free virtual element methods for eigenvalue problems remains relatively limited.  
 Among the relevant studies, for the Laplace eigenvalue problem, \cite{MengWB2022} discussed a stabilization-free low-order virtual element method, while \cite{FolignoBCV2026} provided an in-depth numerical investigation of various stabilized and self-stabilized formulations for the p-version of the virtual element method.
  \cite{MarconMora2025} proposed a stabilization-free low-order virtual element method for non-self-adjoint convection-diffusion eigenvalue problems, demonstrating the robustness of the method through numerical experiments. \cite{AlzabenBD2025} conducted a convergence analysis of virtual element approximations for acoustic vibration eigenvalue problems, showing that optimal convergence can also be achieved without using stabilization terms, with rigorous proofs provided for triangular meshes. \cite{MengGQM2026} presented a stabilization-free low-order virtual element method for the Helmholtz transmission eigenvalue problem in anisotropic media and further gave a high-order and high-dimensional stabilization-free virtual element method in the numerical experiment section without theoretical analysis.  
 Recently, \cite{BerroneBF2025} proposed and analyzed a high-order stabilization-free virtual element method for 2D second-order elliptic equations. \\

\indent Based on the aforementioned work, particularly the work in reference \cite{BerroneBF2025}, in this paper we first establish a high-order stabilization-free virtual element discretization scheme for the general second-order elliptic eigenvalue problem. Subsequently, based on the a priori error estimates for the source problem corresponding to the eigenvalue problem and the regularity assumptions, we prove that the discrete solution operator converges in norm to the exact solution operator. Then, using the spectral approximation theory, we derive the optimal a priori error estimates for both the approximation eigenspace and the approximate eigenvalues.
To validate our theoretical analysis, we conduct numerical experiments on meshes of regular quadrilaterals,
convex-concave pentagons, and concave octagons.
The results indicate that our proposed method is effective. In addition, we compare our method with the standard virtual element method (which includes a stabilization term). The numerical experiments show that the approximate eigenvalues obtained by our method are more accurate than those obtained by the standard virtual element method. \\

\indent In this paper, the letter $C$ (with or without subscripts) denotes a generic positive constant independent of the mesh size $h$, whose specific value may vary depending on the context. The inequality $a \leq Cb$ is abbreviated as $a \lesssim b$.

\section{Eigenvalue problems and their high-order stabilization-free virtual element approximations}
\indent Let $\Omega \in \mathbb{R}^2$ be a bounded open domain with a Lipschitz boundary $\partial\Omega$. Let $H^\xi(D)$ denote the standard complex Sobolev space of order $\xi$ defined on $D$, equipped with the norm $\|\cdot\|_{\xi,D}$ and the seminorm $|\cdot|_{\xi,D}$. In particular, $H^0(D)=L^2(D)$, and the inner product in $L^2(D)$ is given by $(\cdot,\cdot)_{0,D}=\int_D u \overline{v}$. For the sake of simplicity, the subscript $D$ will be omitted when $D=\Omega$. Herein, we consider the following general eigenvalue problem: find $\lambda \in \mathbb{C}$ and a non-trivial function $u$ such that
\begin{eqnarray}
\begin{cases}
\nabla\cdot(-\mathbf{\mathcal{K}} \nabla u + \boldsymbol{\beta} u )+ \gamma u=\lambda u,&\quad \text{in}\ \Omega,\\
\qquad \qquad u=0,&\quad \text{on}\ \partial \Omega, \label{a2.1}
\end{cases}
\end{eqnarray}
\noindent where $\boldsymbol{\beta} \in [L^\infty(\Omega)]^2$ with $\nabla \cdot \boldsymbol{\beta}=0$, and $\gamma \in L^\infty(\Omega)$ and non negative on $\Omega$. We assume that there exist constants $\mathcal{K}_0$ and $\mathcal{K}_1$ such that the symmetric diffusion tensor $\mathbf{\mathcal{K}} \in [L^\infty(\Omega)]^{2\times 2}$ satisfies 
$$\mathcal{K}_0 |\boldsymbol{v}|^2 \leq \overline{\boldsymbol{v}}\cdot(\mathbf{\mathcal{K}}(\boldsymbol{x}) \boldsymbol{v}) \leq \mathcal{K}_1 |\boldsymbol{v}|^2,\qquad \forall \boldsymbol{v} \in \mathbb{C}^2,\ \forall \boldsymbol{x}=(x_1,x_2) \in \Omega, $$ 
where the symbol $|\cdot|$ denotes the Euclidean norm. We further assume that $\mathbf{\mathcal{K}}$, $\boldsymbol{\beta}$ and $\gamma$ are piecewise sufficiently smooth.\\
\indent By Green's formula, the variational formulation of \eqref{a2.1} can be readily derived as follows: find $(\lambda,u)\in \mathbb{C} \times H_0^1(\Omega)$ such that
\begin{eqnarray}
\mathcal{B}(u,v):=a(u,v)+b(u,v)+c(u,v)=\lambda(u,v),\quad \forall ~v \in ~ H_0^1(\Omega), \label{a2.2}
\end{eqnarray}
where
\begin{eqnarray}
a(u,v):=\int_\Omega (\mathbf{\mathcal{K}} \nabla u)\cdot \nabla \overline{v}, ~ b(u,v):=\int_\Omega (\boldsymbol{\beta} \cdot \nabla u) \overline{v},~
c(u,v):=\int_\Omega \gamma u \overline{v},~ \forall u,v \in ~ H_0^1(\Omega). \label{a2.3}
\end{eqnarray}
\noindent For $\forall f \in L^2(\Omega)$, the source problem corresponding to \eqref{a2.2} reads as: find $w \in H_0^1(\Omega)$ such that
\begin{eqnarray}
\mathcal{B}(w,v)=(f,v),\quad \forall ~v \in ~ H_0^1(\Omega). \label{a2.4}
\end{eqnarray}
From \cite{CangianiMS2017,BeiraoBM2016}, we know that the problem \eqref{a2.4} admits a unique solution. 
We define the solution operator associated with problem \eqref{a2.4}:
\begin{eqnarray}
\mathbb{K}:L^2(\Omega) \to H_0^1(\Omega), \nonumber \\
f \mapsto \mathbb{K}f:=w, \label{a2.5}
\end{eqnarray}
and from \cite{BernardiV2000,Grisvard2011}, it follows that for any $f \in L^2(\Omega)$, $\mathbb{K}f \in H^{1+r}(\Omega)$ with $0\leq r \leq 1$, and the estimate holds:
\begin{eqnarray}
\|\mathbb{K}f\|_{1+r} \leq C \|f\|_0. \label{a2.6}
\end{eqnarray}

\indent The problem \eqref{a2.2} is non-self-adjoint, and its adjoint eigenvalue problem is given by: find $(\lambda^*,u^*)\in (\mathbb{C} \times H_0^1(\Omega))$ such that
\begin{eqnarray}
\mathcal{B}(v,u^*)=\overline{\lambda^*}(v,u^*),\quad \forall ~v \in ~ H_0^1(\Omega), \label{a2.7}
\end{eqnarray}
with the original eigenvalue and the adjoint eigenvalue related by $\lambda=\overline{\lambda^*}$. For any $g \in L^2(\Omega)$, the adjoint source problem corresponding to \eqref{a2.4} is given by: find $w^* \in H_0^1(\Omega)$ satisfying
\begin{eqnarray}
\mathcal{B}(v,w^*)=(v,g),\quad \forall ~v \in ~ H_0^1(\Omega). \label{a2.8}
\end{eqnarray}
We similarly define the solution operator associated with \eqref{a2.8}:
\begin{eqnarray}
\mathbb{K}^* :L^2(\Omega) \to H_0^1(\Omega) ,\nonumber \\
g \mapsto \mathbb{K}^*g:=w^*. \label{a2.9}
\end{eqnarray}
\indent Let $\mathcal{T}_h$ be a conforming polygonal partition of $\Omega$. For $\forall E \in \mathcal{T}_h$, let $N_E$ and $h_E$ denote the number of vertices and the diameter of element $E$, respectively. Let $h=\max_{E \in \mathcal{T}_h} h_E$ be the mesh size, $\mathcal{E}_h$ denote the set of edges $e$ of $\mathcal{T}_h$, and $h_e$ be the length of edge $e$. Referring to \cite{BeiraoBC2013}, we make the following assumption on the mesh.
\begin{assumption} \label{ass3.1}
For $\forall E \in \mathcal{T}_h$, there exists a constant $C_\mathcal{T}$ such that
\begin{itemize}
    \itemsep 0mm
    \item Each element $E$ is star-shaped with respect to a ball of radius $\geq C_\mathcal{T} h_E$;
    \item The length $h_e$ of any edge $e$ of $\partial E$ satisfies $h_e \geq C_\mathcal{T} h_E$.
\end{itemize}
\end{assumption}
It is worth noting that the above conditions imply that the number of vertices of each polygon $E$ is uniformly bounded, i.e., there exists a constant $N_{max}$ such that $\forall E \in \mathcal{T}_h$, $N_{E} \leq N_{max}$.\\
\indent Let $\mathbb{P}_k(E)$ be the space of polynomials of degree at most $k$ defined on element $E$, and $\mathbb{P}_k(e)$ be defined analogously to $\mathbb{P}_k(E)$. First, we introduce the elliptic projection $\Pi_k^{\nabla,E}: H^1(E) \to \mathbb{P}_k(E)$, which satisfies
\begin{align}
&(\nabla(\Pi_k^{\nabla,E}v -v), \nabla p)_{0,E}=0,\quad \forall v \in H^1(E), ~ p \in \mathbb{P}_k(E), \label{a3.1} \\
&\frac{1}{|E|}\int_E (\Pi_k^{\nabla,E}v -v)=0, \quad k\geq2,\quad \forall v \in H^1(E). \label{a3.2}
\end{align}
In addition, we introduce the $L^2$ projection $\Pi_k^{0,E}: L^2(E) \to \mathbb{P}_k(E)$, which satisfies
\begin{align}
&(\Pi_k^{0,E}v -v, p)_{0,E}=0,\quad \forall v \in L^2(E), ~ p \in \mathbb{P}_k(E). \label{a3.3}
\end{align}
Then, we define the local virtual element space of order $k~(k\geq2)$ as follows:
\begin{align}
\mathcal{V}^E_{h,k}:=\{ v_h \in H^1(E) : & \Delta v_h |_E \in \mathbb{P}_k(E),~ v_h |_{e} \in \mathbb{P}_k(e), \forall e \in \partial E, ~v_h \in C^0(\partial E),\nonumber \\
& (\Pi_k^{\nabla,E} v_h - v_h, p)_{0,E} = 0, \forall p \in \mathbb{P}_k (E)\setminus \mathbb{P}_{k-2} (E)\}. \label{a3.4}
\end{align}
By gluing the local virtual element spaces given in \eqref{a3.4}, we obtain the global virtual element space:
\begin{eqnarray}
\mathcal{V}_{h,k} := \{ v_h \in H_0^1(\Omega) : v_h |_E \in \mathcal{V}^E_{h,k}, \forall E \in \mathcal{T}_h \}. \label{a3.5}
\end{eqnarray}
Referring to \cite{AhmadAB2013}, we specify the degrees of freedom of $\mathcal{V}^E_{h,k}$ as follows:
\begin{itemize} \itemsep 0mm
    \item The values of $v_h$ at $N_E$ vertices of polygon $E$;
\end{itemize}
\noindent For $k\geq2$,
\begin{itemize} \itemsep 0mm
    \item For $i=1,\cdots,k-1$, $\forall e \in \partial E$, the edge moments $h_e^{-1}(v_h,m_i)_{0,e}$, where $m_i$ are scaled monomials of $\mathbb{P}_{k-2}(e)$;
    \item For $i=1,\cdots,k(k-1)/2$, the interior moments $|E|^{-1}(v_h,m_i)_{0,E}$, where $m_i$ are scaled monomials of $\mathbb{P}_{k-2}(E)$.
\end{itemize}
\indent To introduce the stabilization-free virtual element approximation, we adopt the definition of the high-order polynomial projection operator $\Pi_\mathcal{P}^{0,E} \nabla$ from \cite{BerroneBOFA2026,BerroneBF2025}. For $\forall E \in \mathcal{T}_h$, let $\ell$ be a given natural number, and define
\begin{eqnarray}
\mathcal{P}_{k,\ell}:=[\mathbb{P}_{k-1}(E)]^2 \oplus \textbf{curl} \left(\mathbb{P}_{k+\ell}(E)\setminus \mathbb{P}_{k}(E)\right)
= \boldsymbol{x} \mathbb{P}_{k-2}(E) \oplus \textbf{curl} \left(\mathbb{P}_{k+\ell}(E)\right),\label{a3.6}
\end{eqnarray}
where $\textbf{curl} (p)=\left(\partial p/\partial x_2,-\partial p/\partial x_1\right)$. Next, we define the $L^2$-orthogonal projection $\Pi_\mathcal{P}^{0,E} \nabla: H^1(E) \to \mathcal{P}_{k,\ell}$, which satisfies
\begin{align}
&\left(\Pi_\mathcal{P}^{0,E} \nabla v - \nabla v, \boldsymbol{p}\right)_{0,E}=0,\quad \forall v \in H^1(E), ~ \boldsymbol{p} \in \mathcal{P}_{k,\ell}. \label{a3.7}
\end{align}
\indent Note that the divergence operator $div$ mapping from $ \boldsymbol{x} \mathbb{P}_{k-2}(E)$ to $\mathbb{P}_{k-2}(E)$ is surjective. For $\forall v_h \in \mathcal{V}_{h,k}^E$ and $\boldsymbol{p} \in \mathcal{P}_{k,\ell}$,
it follows from \eqref{a3.6} that $\boldsymbol{p}=\boldsymbol{p}_1 + \boldsymbol{p}_2$ where $\boldsymbol{p}_1 \in \boldsymbol{x} \mathbb{P}_{k-2}(E)$ and $\boldsymbol{p}_2 \in \textbf{curl} \left(\mathbb{P}_{k+\ell}(E)\right)$. Since $div(\boldsymbol{p}_2)=0$,
applying the divergence theorem to \eqref{a3.7}, we obtain
\begin{align*}
\left(\Pi_\mathcal{P}^{0,E} \nabla v_h, \boldsymbol{p}\right)_{0,E} = \left(\nabla v_h, \boldsymbol{p}\right)_{0,E}=-\int_{E} v_h div(\boldsymbol{p}_1) + \int_{\partial E} v_h \left(\boldsymbol{p}\cdot \boldsymbol{n}^{\partial E}\right),
\end{align*}
where $\boldsymbol{n}^{\partial E}$ denotes the unit outward normal vector of $\partial E$. As $div(\boldsymbol{p}_1) \in \mathbb{P}_{k-2}(E)$, it can be easily seen from the choice of degrees of freedom that the projection operator $\Pi_\mathcal{P}^{0,E} \nabla$ is computable solely by using the degrees of freedom of $v_h$.\\
\indent Next, for $\forall E \in \mathcal{T}_h$, we define
\begin{eqnarray}
\mathcal{K}_E^\vee:=\inf_{\boldsymbol{x} \in E} \sup_{\boldsymbol{v}\in \mathbb{C}^2} \frac{\overline{\boldsymbol{v}}\cdot(\mathbf{\mathcal{K}}(\boldsymbol{x}) \boldsymbol{v})}{|\boldsymbol{v}|^2},~
\mathcal{K}_E^\wedge:=\sup_{\boldsymbol{x} \in E} \sup_{\boldsymbol{v}\in \mathbb{C}^2} \frac{\overline{\boldsymbol{v}}\cdot(\mathbf{\mathcal{K}}(\boldsymbol{x}) \boldsymbol{v})}{|\boldsymbol{v}|^2}, \label{a3.8}
\end{eqnarray}
as well as the sesquilinear form $\mathcal{B}_h(\cdot,\cdot)$ as follows:
\begin{eqnarray}
\mathcal{B}_h(u_h,v_h):= \sum_{E \in \mathcal{T}_h}a_h^E(u_h,v_h)+b_h^E(u_h,v_h)+c_h^E(u_h,v_h),\quad \forall ~u_h,v_h \in ~ \mathcal{V}_{h,k}, \label{a3.9}
\end{eqnarray}
where
\begin{align}
&a_h^E(u_h,v_h):=\left(\mathbf{\mathcal{K}} \Pi_\mathcal{P}^{0,E} \nabla u_h,\Pi_\mathcal{P}^{0,E} \nabla v_h\right)_{0,E},&\quad \forall ~u_h,v_h \in~\mathcal{V}_{h,k}^E, \label{a3.10}\\
&b_h^E(u_h,v_h):=\left(\boldsymbol{\beta} \cdot \Pi_{k-1}^{0,E} \nabla u_h,\Pi_k^{0,E} v_h\right)_{0,E},&\quad \forall ~u_h,v_h \in~\mathcal{V}_{h,k}^E, \label{a3.11}\\
&c_h^E(u_h,v_h):=\left(\gamma \Pi_k^{0,E} u_h,\Pi_k^{0,E} v_h\right)_{0,E},&\quad \forall ~u_h,v_h \in~\mathcal{V}_{h,k}^E. \label{a3.12}
\end{align}
Then, the stabilization-free virtual element discretization of \eqref{a2.2} is to find $(\lambda_h,u_h)\in \mathbb{C} \times \mathcal{V}_{h,k}$ such that
\begin{eqnarray}
\mathcal{B}_h(u_h,v_h)= \lambda_h \sum_{E \in \mathcal{T}_h}  \left(\Pi_k^{0,E} u_h,\Pi_k^{0,E} v_h\right)_{0,E},\quad \forall v_h \in~\mathcal{V}_{h,k}. \label{a3.13}
\end{eqnarray}
And the stabilization-free virtual element discretization of \eqref{a2.3} is to find $w_h \in \mathcal{V}_{h,k}$ such that
\begin{eqnarray}
\mathcal{B}_h(w_h,v_h)= \sum_{E \in \mathcal{T}_h}  \left(f,\Pi_k^{0,E} v_h\right)_{0,E},\quad \forall v_h \in~\mathcal{V}_{h,k}. \label{a3.14}
\end{eqnarray}
\indent For $\forall E \in \mathcal{T}_h$ and $\forall v_h \in~\mathcal{V}_{h,k}^E$, let $a(\cdot,\cdot)$, $b(\cdot,\cdot)$ and $c(\cdot,\cdot)$ restricted to an element $E$ be denoted by $a^E(\cdot,\cdot)$, $b^E(\cdot,\cdot)$ and $c^E(\cdot,\cdot)$ respectively. Then we define
\begin{eqnarray*}
 |\| v_h \||_{E}^2:=a^E(v_h,v_h)+c_h^E(v_h,v_h):=\left\| \sqrt{\mathbf{\mathcal{K}}} \nabla v_h\right\|_{0,E}^2 + \left\|\sqrt{\gamma} \Pi_k^{0,E} v_h\right\|_{0,E}^2,
\end{eqnarray*}
and for $\forall v_h \in~\mathcal{V}_{h,k}$, define
\begin{eqnarray*}
 |\| v_h \||^2:= \sum_{E \in \mathcal{T}_h} |\| v_h \||_{E}^2.
\end{eqnarray*}
\begin{assumption} \label{ass3.2}
Let $\mathbb{P}_{k-1}^0(\partial E):=\{\pi: \pi|_e \in  \mathbb{P}_{k-1}(e),\ \forall e \in \partial E,\ \int_{\partial E} \pi =0 \}$. Assume that $\ell$ is the minimal integer such that any polynomial $q \in \mathbb{P}_{k+\ell}(E)$ can be uniquely determined by a set of degrees of freedom, which include $k N_E-1$ distinct moments $(1/|\partial E|)(q,\pi_i)_{0,\partial E}$ where $\pi_i$ are scaled polynomial bases of the space $\mathbb{P}_{k-1}^0(\partial E)$.
\end{assumption}
\indent According to Theorems 2 and 3 in \cite{BerroneBF2025}, we know that the sesquilinear form $\mathcal{B}_h(\cdot,\cdot)$ is bounded and coercive, that is, when the mesh size $h$ is sufficiently small, under Assumptions \ref{ass3.1} and \ref{ass3.2}, then there hold
\begin{align}
&\mathcal{B}_h(u_h,v_h) \lesssim |\| u_h \|| |\| v_h \||, &\forall u_h,v_h \in~\mathcal{V}_{h,k}, \label{a3.15}\\
&\mathcal{B}_h(v_h,v_h) \lesssim  |\| v_h \||^2,  &\forall v_h \in~\mathcal{V}_{h,k} \label{a3.16}.
\end{align}
From \eqref{a3.15} and \eqref{a3.16}, by the Lax-Milgram Theorem, we conclude that the discrete source problem \eqref{a3.14} admits a unique solution. 
For $\forall f \in L^2(\Omega)$, we thus define the discrete solution operator:
\begin{eqnarray}
\mathbb{K}_h:L^2(\Omega) \to \mathcal{V}_{h,k} \subset H_0^1(\Omega),\nonumber \\
f \mapsto \mathbb{K}_h f:=w_h. \label{a3.17}
\end{eqnarray}
\indent Similarly, the adjoint eigenvalue problem of the discrete eigenvalue problem \eqref{a3.13} is to find $(\lambda_h^*,u_h^*)\in \mathbb{C} \times \mathcal{V}_{h,k}$ such that
\begin{eqnarray}
\mathcal{B}_h(v_h,u_h^*)=\overline{\lambda_h^*} \sum_{E \in \mathcal{T}_h}  (\Pi_k^{0,E} v_h,\Pi_k^{0,E} u_h^*)_{0,E},\quad \forall ~v_h \in ~ \mathcal{V}_{h,k}, \label{a3.18}
\end{eqnarray}
with the discrete eigenvalue and discrete adjoint eigenvalue related by $\lambda_h=\overline{\lambda_h^*}$. For any $g \in L^2(\Omega)$, the stabilization-free virtual element discretization of \eqref{a2.8} is to find $w_h^* \in \mathcal{V}_{h,k}$ such that
\begin{eqnarray}
\mathcal{B}_h(v_h,w_h^*)=\sum_{E \in \mathcal{T}_h}  (\Pi_k^{0,E} v_h,g)_{0,E},\quad \forall ~v_h \in ~ \mathcal{V}_{h,k}, \label{a3.19}
\end{eqnarray}
whence the discrete solution operator is defined as
\begin{eqnarray}
\mathbb{K}_h^* :L^2(\Omega) \to \mathcal{V}_{h,k} \subset H_0^1(\Omega), \nonumber \\
g \mapsto \mathbb{K}_h^*g:=w_h^*. \label{a3.20}
\end{eqnarray}

\section{A priori error analysis}
\indent In this section, we focus primarily on the a priori error estimate analysis for the eigenvalue problem \eqref{a3.13}. To apply the spectral approximation theory, we first prove that the discrete solution operator $\mathbb{K}_h$ converges to the exact solution operator $\mathbb{K}$ in norm.\\
\indent As stated in Theorems 5 and 6 of \cite{BerroneBF2025}, the following a priori error estimates hold for the source problem \eqref{a2.4} and its discrete counterpart \eqref{a3.14}.
\begin{lemma} \label{lem3.3}
Let $f \in H^{\max\{0,s-1\}}(\Omega)$, and $w \in H^{1+s}(\Omega) \cap H_0^1(\Omega)$ $(0 \leq r\leq s \leq k)$ be the solution
of \eqref{a2.4}. Then the unique solution $w_h$ of \eqref{a3.14} satisfies the error estimate
\begin{eqnarray}
|\| w-w_h \|| \lesssim h^{s} (|w|_{1+s}+|f|_{\max\{0,s-1\}}); \label{a3.21}
\end{eqnarray}
furthermore, it holds that
\begin{eqnarray}
\| w-w_h \|_{0} \lesssim h^{r+s} (|w|_{1+s}+|f|_{\max\{0,s-1\}}). \label{a3.22}
\end{eqnarray}
\end{lemma}
\begin{remark}\label{rem3.4}  For the solution $w^*$ of the adjoint source problem \eqref{a2.8} and the solution $w_h^*$ of the discrete adjoint source problem \eqref{a3.19}, similar a priori error estimates as those in Lemma \ref{lem3.3} as hold.
\end{remark}
From \eqref{a2.4} and \eqref{a2.5}, as well as \eqref{a3.14} and \eqref{a3.17}, we have $w=\mathbb{K}f$ and $w_h=\mathbb{K}_h f$. As the mesh size $h \to 0$, combining \eqref{a2.6} and \eqref{a3.21} yields $|\| \mathbb{K}_h-\mathbb{K} \|| \to 0$, and combining \eqref{a2.6} and \eqref{a3.22} yields $\| \mathbb{K}_h-\mathbb{K} \|_{0} \to 0$.\\
\indent Let $\lambda_j$ be an eigenvalue of \eqref{a2.2} with multiplicity $q$, the eigenvalues are sorted in ascending order based on their modulus, i.e., $\lambda_{j-1} \neq \lambda_{j} = \lambda_{j+1} = \cdots = \lambda_{j+q-1} \neq \lambda_{j+q}$, and denote the corresponding eigenspace by $M(\lambda_j)$. For the adjoint eigenvalue $\lambda_j^*$ of \eqref{a2.2}, we define $M^*(\lambda_j^*)$ analogously to $M(\lambda_j)$. According to  spectral approximation theory \cite{BaO1991}, we know that as the mesh size $h \to 0$, there exist $q$ discrete eigenvalues $\lambda_{j,h},\cdots,\lambda_{j+q-1,h}$ of \eqref{a3.13} that converge to $\lambda_j$. Let $M_h(\lambda_j)$ be the direct sum of the discrete eigenspaces corresponding to the eigenvalues $\lambda_{j,h},\cdots,\lambda_{j+q-1,h}$, and set $\widehat{\lambda_{j,h}}=1/q\sum_{i=1}^q \lambda_{j+i-1,h}$.\\
\indent For two closed subspaces ${\mathcal{M}}$ and ${\mathcal{N}}$ of $H_0^1(\Omega)$, we recall the gap between ${\mathcal{M}}$ and ${\mathcal{N}}$ with respect to the norm $|\| \cdot \||$:
$$\quad \widehat{\delta}_1(\mathcal{M},\mathcal{N})=\max(\delta_1(\mathcal{M},\mathcal{N}),\delta _1(\mathcal{N},\mathcal{M})).
$$
where
$$
\delta_1(\mathcal{M},\mathcal{N}) = \sup_{\Phi \in \mathcal{M}, |\| \Phi|\|=1}
\inf_{\Psi \in\mathcal{N}} |\| \Phi -  \Psi |\|.
$$
Analogously, we can define the gap $\widehat{\delta}_0(\mathcal{M},\mathcal{N})$ between two subspaces $\mathcal{M}$ and $\mathcal{N}$ of $L^2(\Omega)$ with respect to the norm $\|\cdot\|_0$.
Furthermore, we also recall the projection error estimate on the polynomial space $\mathbb{P}_{k}(E)$.
\begin{proposition} (see Proposition 4.2 in \cite{BeiraoBC2013} and Proposition 4.1 in \cite{MoraRiveraRodriguez2015}) \label{pro3.5}
If the first condition in Assumption \ref{ass3.1} holds, then there exists a constant $C$ independent of $h$ such that for all $s$ satisfying $0 \leq s \leq k$ and each $v \in H^{1+s}(E)$, there exists $v_\pi \in \mathbb{P}_{k}(E)$ such that
\begin{eqnarray}
\|v-v_\pi\|_{0,E} + h_E|v-v_\pi|_{1,E} \leq Ch_E^{1+s} \|v\|_{1+s,E}. \label{a3.23}
\end{eqnarray}
\end{proposition}
\indent Now, by spectral approximation theory \cite{BaO1991}, we establish the following a priori error estimates for the approximation spaces and approximate eigenvalues.
\begin{theorem} \label{th3.1}
Let $\lambda_j$ be an eigenvalue of \eqref{a2.2} with multiplicity $q$, and let $\lambda_{j,h}$ be an eigenvalue of \eqref{a3.13} that converges to $\lambda_j$. Assume that $M(\lambda_j) \subset H^{1+s}(\Omega)$ with $k \geq s \geq r$. Then the following estimates hold:
\begin{align}
&\widehat{\delta}_1(M(\lambda_j),M_h(\lambda_j))  \lesssim h^{s}, \label{a3.25}\\
&\widehat{\delta}_0(M(\lambda_j),M_h(\lambda_j))   \lesssim h^{r+s}, \label{a3.26}\\
& | \lambda_j - \widehat{\lambda_{j,h}} | \lesssim h^{2s}. \label{a3.27}
\end{align}
\end{theorem}
\begin{proof}
As $h \to 0$, since $|\| \mathbb{K}_h - \mathbb{K} |\|  \to 0$ and $\| \mathbb{K}_h - \mathbb{K} \|_{0}  \to 0$, by Theorem 7.1 in \cite{BaO1991}, we obtain
\begin{align}
\widehat{\delta}_1(M(\lambda_j),M_h(\lambda_j)) & \lesssim |\| (\mathbb{K}_h - \mathbb{K})|_{M(\lambda_j)} |\| ,\label{a3.28} \\
\widehat{\delta}_0(M(\lambda_j),M_h(\lambda_j)) & \lesssim \| (\mathbb{K}_h - \mathbb{K})|_{M(\lambda_j)} \|_0. \label{a3.29}
\end{align}
For any $u \in M(\lambda_j)$, from \eqref{a3.21} and \eqref{a3.22}, we derive the following estimates respectively:
\begin{align}
|\| (\mathbb{K}_h - \mathbb{K})u |\| \lesssim h^{s} (|\mathbb{K}u|_{1+s}+|u|_{\max\{0,s-1\}}),\label{a3.30} \\
\| (\mathbb{K}_h - \mathbb{K})u \|_0 \lesssim h^{r+s} (|\mathbb{K}u|_{1+s}+|u|_{\max\{0,s-1\}}) .\label{a3.31}
\end{align}
The estimate \eqref{a3.25} follows from combining \eqref{a3.28} and \eqref{a3.30}, and the estimate \eqref{a3.26} follows from combining \eqref{a3.29} and \eqref{a3.31}.\\
We next prove \eqref{a3.27}. Let $f_j,\cdots,f_{j+q-1}$ be the basis functions of $M(\lambda_j)$, and $f_j^*,\cdots,f_{j+q-1}^*$ be the basis functions of $M^*(\lambda_j^*)$, satisfying $(f_\imath,f_\jmath^*)_0=\delta_{\imath\jmath}$ for $\imath,\jmath=j,j+q-1$. By Theorem 7.2 in \cite{BaO1991}, we have
\begin{align}
|\lambda_j - \widehat{\lambda_{j,h}} | \leq \frac{1}{q} \sum_{i=j}^{j+q-1}|((\mathbb{K}-\mathbb{K}_h)f_i,f_i^*)|+C\|(\mathbb{K}-\mathbb{K}_h)|_{M(\lambda_j)}\|_0 \|(\mathbb{K}^*-\mathbb{K}_h^*)|_{M^*(\lambda_j^*)}\|_0. \label{a3.32}
\end{align}
Combining Remark \ref{rem3.4}, note that the second term on the right-hand side of \eqref{a3.32} is a higher-order infinitesimal compared with the first term. Thus, we only need to bound the first term on the right-hand side of \eqref{a3.32}. In fact, from \eqref{a2.5}, \eqref{a2.9} and \eqref{a3.20}, we deduce that
\begin{align}
((\mathbb{K}-\mathbb{K}_h)f_i,f_i^*) &=(f_i,(\mathbb{K}^*-\mathbb{K}_h^*)f_i^*)=\mathcal{B}(\mathbb{K} f_i,(\mathbb{K}^*-\mathbb{K}_h^*)f_i^*) \nonumber \\
&=\mathcal{B}((\mathbb{K}-\mathbb{K}_h) f_i,(\mathbb{K}^*-\mathbb{K}_h^*)f_i^*) + \mathcal{B}(\mathbb{K}_h f_i,(\mathbb{K}^*-\mathbb{K}_h^*)f_i^*) \nonumber \\
&=\mathcal{B}((\mathbb{K}-\mathbb{K}_h) f_i,(\mathbb{K}^*-\mathbb{K}_h^*)f_i^*) + \mathcal{B}(\mathbb{K}_h f_i,\mathbb{K}^*f_i^*) -\mathcal{B}(\mathbb{K}_h f_i,\mathbb{K}_h^* f_i^*) \nonumber \\
&=\mathcal{B}((\mathbb{K}-\mathbb{K}_h) f_i,(\mathbb{K}^*-\mathbb{K}_h^*)f_i^*) + (\mathbb{K}_h f_i,f_i^*) -\mathcal{B}(\mathbb{K}_h f_i,\mathbb{K}_h^* f_i^*) \nonumber \\
&=\mathcal{B}((\mathbb{K}-\mathbb{K}_h) f_i,(\mathbb{K}^*-\mathbb{K}_h^*)f_i^*) + \mathcal{B}_h(\mathbb{K}_h f_i,\mathbb{K}_h^* f_i^*) - \sum_{E \in \mathcal{T}_h}(\Pi_k^{0,E} \mathbb{K}_h f_i,f_i^*)_{0,E}\nonumber \\
&\quad  + (\mathbb{K}_h f_i,f_i^*) -\mathcal{B}(\mathbb{K}_h f_i,\mathbb{K}_h^* f_i^*) \nonumber \\
&=\mathcal{B}((\mathbb{K}-\mathbb{K}_h) f_i,(\mathbb{K}^*-\mathbb{K}_h^*)f_i^*) + \mathcal{B}_h(\mathbb{K}_h f_i,\mathbb{K}_h^* f_i^*)-\mathcal{B}(\mathbb{K}_h f_i,\mathbb{K}_h^* f_i^*) \nonumber \\
&\quad  + \sum_{E \in \mathcal{T}_h}(\mathbb{K}_h f_i-\Pi_k^{0,E} \mathbb{K}_h f_i,f_i^*)_{0,E} \nonumber \\
&:= I + I\!I + I\!I\!I. \label{a3.33}
\end{align}
For the term $I$, combining the boundedness of $\mathcal{B}(\cdot,\cdot)$ with \eqref{a3.21}, we obtain
\begin{align}
|I| \lesssim h^{2s}.\label{a3.34}
\end{align}
For the term $I\!I$, using the triangle inequality we obtain
\begin{align}
|I\!I|\leq\sum_{E \in \mathcal{T}_h} & (|a^E(\mathbb{K}_h f_i,\mathbb{K}^*_h f_i^*)-a_h^E(\mathbb{K}_h f_i,\mathbb{K}^*_h f_i^*)| + |b^E(\mathbb{K}_h f_i,\mathbb{K}^*_h f_i^*)-b_h^E(\mathbb{K}_h f_i,\mathbb{K}^*_h f_i^*)| \nonumber \\
&+|c^E(\mathbb{K}_h f_i,\mathbb{K}^*_h f_i^*)-c_h^E(\mathbb{K}_h f_i,\mathbb{K}^*_h f_i^*)|)
.\label{a3.35}
\end{align}
Herein, for any $v \in H^{1+s}(E)$, since $\Pi_{k-1}^{0,E} \nabla v \in [\mathbb{P}_{k-1}(E)]^2 \subset \mathcal{P}_{k,\ell}$, by using \eqref{a3.23}, we derive
\begin{align}
\|\nabla v - \Pi_\mathcal{P}^{0,E} \nabla v \|_{0,E} \leq \|\nabla v - \Pi_{k-1}^{0,E} \nabla v \|_{0,E} \lesssim h_E^{s} |\nabla v|_{s,E}
\lesssim h_E^{s} |v|_{1+s,E}. \label{a3.36}
\end{align}
Combining \eqref{a3.7}, \eqref{a3.8}, \eqref{a3.21}, \eqref{a3.36} and Lemma 5.3 in \cite{BeiraoBM2016}, we obtain
\begin{align}
& \sum_{E \in \mathcal{T}_h} |a^E(\mathbb{K}_h f_i,\mathbb{K}^*_h f_i^*)-a_h^E(\mathbb{K}_h f_i,\mathbb{K}^*_h f_i^*)| \nonumber \\
&= \sum_{E \in \mathcal{T}_h} |(\mathbf{\mathcal{K}} \Pi_\mathcal{P}^{0,E} \nabla (\mathbb{K}_h f_i),\Pi_\mathcal{P}^{0,E} \nabla (\mathbb{K}^*_h f_i^*))_{0,E}
    - (\mathbf{\mathcal{K}} \nabla \mathbb{K}_h f_i, \nabla \mathbb{K}^*_h f_i^*)_{0,E}  | \nonumber \\
&\leq \sum_{E \in \mathcal{T}_h} (\|\mathbf{\mathcal{K}} \nabla (\mathbb{K}_h f_i) - \Pi_\mathcal{P}^{0,E} \mathbf{\mathcal{K}} \nabla (\mathbb{K}_h f_i) \|_{0,E}
     \|\nabla \mathbb{K}^*_h f_i^*- \Pi_\mathcal{P}^{0,E} \nabla (\mathbb{K}^*_h f_i^*)\|_{0,E} \nonumber \\
 & \qquad +  \|\nabla (\mathbb{K}_h f_i) - \Pi_\mathcal{P}^{0,E} \nabla (\mathbb{K}_h f_i) \|_{0,E}
     \|\mathbf{\mathcal{K}} \nabla \mathbb{K}^*_h f_i^*- \Pi_\mathcal{P}^{0,E} \mathbf{\mathcal{K}} \nabla (\mathbb{K}^*_h f_i^*)\|_{0,E} \nonumber \\
  & \qquad +  \mathcal{K}_E^\wedge \|\nabla (\mathbb{K}_h f_i) - \Pi_\mathcal{P}^{0,E} \nabla (\mathbb{K}_h f_i) \|_{0,E}
     \|\nabla \mathbb{K}^*_h f_i^*- \Pi_\mathcal{P}^{0,E} \nabla (\mathbb{K}^*_h f_i^*)\|_{0,E} )  \nonumber \\
& \lesssim  \sum_{E \in \mathcal{T}_h} \{ (\|\mathbf{\mathcal{K}} \nabla ((\mathbb{K}_h-\mathbb{K}) f_i) \|_{0,E} + \|\mathbf{\mathcal{K}} \nabla \mathbb{K} f_i- \Pi_{k-1}^{0,E}(\mathbf{\mathcal{K}} \nabla \mathbb{K} f_i) \|_{0,E} ) \|\nabla \mathbb{K}^*_h f_i^*- \Pi_\mathcal{P}^{0,E} \nabla (\mathbb{K}^*_h f_i^*)\|_{0,E} \nonumber \\
 & \qquad + ( \frac{1}{\sqrt{\mathcal{K}_E^\vee}}\| \sqrt{\mathbf{\mathcal{K}}} \nabla ((\mathbb{K}_h-\mathbb{K}) f_i) \|_{0,E} + \|\nabla (\mathbb{K} f_i) - \Pi_{k-1}^{0,E} \nabla (\mathbb{K} f_i) \|_{0,E}  ) \|\mathbf{\mathcal{K}} \nabla \mathbb{K}^*_h f_i^*- \Pi_\mathcal{P}^{0,E} \mathbf{\mathcal{K}} \nabla (\mathbb{K}^*_h f_i^*)\|_{0,E} \nonumber \\
  & \qquad +
   ( \frac{\mathcal{K}_E^\wedge}{\sqrt{\mathcal{K}_E^\vee}}\| \sqrt{\mathbf{\mathcal{K}}} \nabla ((\mathbb{K}_h-\mathbb{K}) f_i) \|_{0,E} + \mathcal{K}_E^\wedge \|\nabla (\mathbb{K} f_i) - \Pi_{k-1}^{0,E} \nabla (\mathbb{K} f_i) \|_{0,E}  )
  \|\nabla \mathbb{K}^*_h f_i^*- \Pi_\mathcal{P}^{0,E} \nabla (\mathbb{K}^*_h f_i^*)\|_{0,E} 
\}  \nonumber \\ 
& \lesssim  \sum_{E \in \mathcal{T}_h} \{
(\sqrt{\mathcal{K}_E^\wedge} \|\sqrt{\mathbf{\mathcal{K}}} \nabla ((\mathbb{K}_h-\mathbb{K}) f_i) \|_{0,E} + \|\mathbf{\mathcal{K}}\|_{s,\infty,E}h_E^s \| \mathbb{K} f_i \|_{s+1,E} ) \|\nabla \mathbb{K}^*_h f_i^*- \Pi_\mathcal{P}^{0,E} \nabla (\mathbb{K}^*_h f_i^*)\|_{0,E} \nonumber \\
& \qquad + ( \frac{1}{\sqrt{\mathcal{K}_E^\vee}}\| \sqrt{\mathbf{\mathcal{K}}} \nabla ((\mathbb{K}_h-\mathbb{K}) f_i) \|_{0,E} + h_E^s \| \mathbb{K} f_i \|_{s+1,E} ) \|\mathbf{\mathcal{K}} \nabla \mathbb{K}^*_h f_i^*- \Pi_\mathcal{P}^{0,E} \mathbf{\mathcal{K}} \nabla (\mathbb{K}^*_h f_i^*)\|_{0,E} \nonumber \\
& \qquad +
   ( \frac{\mathcal{K}_E^\wedge}{\sqrt{\mathcal{K}_E^\vee}}\| \sqrt{\mathbf{\mathcal{K}}} \nabla ((\mathbb{K}_h-\mathbb{K}) f_i) \|_{0,E} + \mathcal{K}_E^\wedge h_E^s \| \mathbb{K} f_i \|_{s+1,E} )
  \|\nabla \mathbb{K}^*_h f_i^*- \Pi_\mathcal{P}^{0,E} \nabla (\mathbb{K}^*_h f_i^*)\|_{0,E} )  \}   \nonumber \\
& \lesssim  h^{2s}, \label{a3.37}
\end{align}
by using \eqref{a3.3}, \eqref{a3.8}, \eqref{a3.21}, \eqref{a3.22}, \eqref{a3.23} and Lemma 5.3 in \cite{BeiraoBM2016}, we get
\begin{align}
& \sum_{E \in \mathcal{T}_h} |b^E(\mathbb{K}_h f_i,\mathbb{K}^*_h f_i^*)-b_h^E(\mathbb{K}_h f_i,\mathbb{K}^*_h f_i^*)| \nonumber \\
&= \sum_{E \in \mathcal{T}_h} |(\boldsymbol{\beta} \cdot \nabla \mathbb{K}_h f_i, \mathbb{K}^*_h f_i^*)_{0,E} -
(\boldsymbol{\beta} \cdot \Pi_{k-1}^{0,E} \nabla \mathbb{K}_h f_i ,\Pi_k^{0,E} \mathbb{K}^*_h f_i^*)_{0,E} | \nonumber \\
&\leq \sum_{E \in \mathcal{T}_h}\{ \| \Pi_{k-1}^{0,E} (\boldsymbol{\beta} \cdot \nabla \mathbb{K}_h f_i) - \boldsymbol{\beta} \cdot \nabla \mathbb{K}_h f_i\|_{0,E}
\|\mathbb{K}^*_h f_i^* -\Pi_k^{0,E} \mathbb{K}^*_h f_i^*\| \nonumber \\
& \qquad + \|  \Pi_{k-1}^{0,E} (\nabla \mathbb{K}_h f_i) -  \nabla \mathbb{K}_h f_i\|_{0,E}
\| \boldsymbol{\beta} \mathbb{K}^*_h f_i^*- \Pi_{k-1}^{0,E} (\boldsymbol{\beta}  \mathbb{K}^*_h f_i^*)\| \nonumber \\
& \qquad + \|\boldsymbol{\beta}\|_{0,\infty,E} \|  \Pi_{k-1}^{0,E} (\nabla \mathbb{K}_h f_i) -  \nabla \mathbb{K}_h f_i\|_{0,E}
\|  \mathbb{K}^*_h f_i^*- \Pi_{k}^{0,E} ( \mathbb{K}^*_h f_i^*)\| \}\nonumber \\
&\lesssim \sum_{E \in \mathcal{T}_h}\{ (\| \Pi_{k-1}^{0,E} (\boldsymbol{\beta} \cdot \nabla \mathbb{K} f_i) - \boldsymbol{\beta} \cdot \nabla \mathbb{K} f_i\|_{0,E}
+\| (\boldsymbol{\beta} \cdot \nabla (\mathbb{K}-\mathbb{K}_h) f_i \|_{0,E} )
\|\mathbb{K}^*_h f_i^* -\Pi_k^{0,E} \mathbb{K}^*_h f_i^*\| \nonumber \\
& \qquad + (\|  \Pi_{k-1}^{0,E} (\nabla \mathbb{K} f_i) -  \nabla \mathbb{K} f_i\|_{0,E} + \|  \nabla (\mathbb{K}-\mathbb{K}_h) f_i\|_{0,E} )
\| \boldsymbol{\beta} \mathbb{K}^*_h f_i^*- \Pi_{k-1}^{0,E} (\boldsymbol{\beta}  \mathbb{K}^*_h f_i^*)\| \nonumber \\
& \qquad + \|\boldsymbol{\beta}\|_{0,\infty,E} (\|  \Pi_{k-1}^{0,E} (\nabla \mathbb{K} f_i) -  \nabla \mathbb{K} f_i\|_{0,E} + \|  \nabla (\mathbb{K}-\mathbb{K}_h) f_i\|_{0,E} )
\|  \mathbb{K}^*_h f_i^*- \Pi_{k}^{0,E} ( \mathbb{K}^*_h f_i^*)\| \}\nonumber \\
&\lesssim \sum_{E \in \mathcal{T}_h}\{ ( h_E^s \|\boldsymbol{\beta}\|_{s,\infty,E} |\mathbb{K} f_i|_{1+s,E}
+  \frac{\|\boldsymbol{\beta}\|_{0,\infty,E}}{ \sqrt{\mathcal{K}_E^\vee}} \| \sqrt{\mathcal{K}}\nabla (\mathbb{K}-\mathbb{K}_h) f_i \|_{0,E} )
\|\mathbb{K}^*_h f_i^* -\Pi_k^{0,E} \mathbb{K}^*_h f_i^*\| \nonumber \\
& \qquad + ( h_E^s |\mathbb{K} f_i|_{1+s,E} + \frac{1}{\sqrt{\mathcal{K}_E^\vee}} \| \sqrt{\mathcal{K}}\nabla (\mathbb{K}-\mathbb{K}_h) f_i \|_{0,E} )
\| \boldsymbol{\beta} \mathbb{K}^*_h f_i^*- \Pi_{k-1}^{0,E} (\boldsymbol{\beta}  \mathbb{K}^*_h f_i^*)\| \nonumber \\
& \qquad + ( h_E^s \|\boldsymbol{\beta}\|_{0,\infty,E} |\mathbb{K} f_i|_{1+s,E} + \frac{\|\boldsymbol{\beta}\|_{0,\infty,E}}{\sqrt{\mathcal{K}_E^\vee}} \| \sqrt{\mathcal{K}}\nabla (\mathbb{K}-\mathbb{K}_h) f_i \|_{0,E} )
\|  \mathbb{K}^*_h f_i^*- \Pi_{k}^{0,E} ( \mathbb{K}^*_h f_i^*)\| \}\nonumber \\
&\lesssim h^{2s}, \label{a3.38}
\end{align}
and by invoking Lemma 5.3 in \cite{BeiraoBM2016} again, together with \eqref{a3.3}, \eqref{a3.22} and \eqref{a3.23}, we obtain 
\begin{align}
& \sum_{E \in \mathcal{T}_h} |c^E(\mathbb{K}_h f_i,\mathbb{K}^*_h f_i^*)-c_h^E(\mathbb{K}_h f_i,\mathbb{K}^*_h f_i^*)| \nonumber \\
&= \sum_{E \in \mathcal{T}_h} |(\gamma \mathbb{K}_h f_i, \mathbb{K}^*_h f_i^*)_{0,E} -
(\gamma \Pi_k^{0,E} \mathbb{K}_h f_i ,\Pi_k^{0,E} \mathbb{K}^*_h f_i^*)_{0,E} | \nonumber \\
&\leq  \sum_{E \in \mathcal{T}_h} \{
\|\gamma \mathbb{K}_h f_i - \Pi_k^{0,E} (\gamma \mathbb{K}_h f_i) \|_{0,E} \|\mathbb{K}^*_h f_i^*-\Pi_k^{0,E} \mathbb{K}^*_h f_i^* \|_{0,E} \nonumber \\
& \qquad + \|\mathbb{K}_h f_i - \Pi_k^{0,E} \mathbb{K}_h f_i\|_{0,E} \|\gamma \mathbb{K}^*_h f_i^*-\Pi_k^{0,E} (\gamma  \mathbb{K}^*_h f_i^*) \|_{0,E} \nonumber \\
& \qquad + \|\gamma\|_{0,\infty,E}\|\mathbb{K}_h f_i - \Pi_k^{0,E} \mathbb{K}_h f_i\|_{0,E} \|\mathbb{K}^*_h f_i^*-\Pi_k^{0,E} \mathbb{K}^*_h f_i^*\|_{0,E} \} \nonumber \\
&\lesssim  \sum_{E \in \mathcal{T}_h} \{
(\|\gamma (\mathbb{K}-\mathbb{K}_h) f_i \|_{0,E} + \|\gamma \mathbb{K} f_i - \Pi_k^{0,E} (\gamma \mathbb{K} f_i) \|_{0,E})\|\mathbb{K}^*_h f_i^*-\Pi_k^{0,E} \mathbb{K}^*_h f_i^* \|_{0,E} \nonumber \\
& \qquad + (\|(\mathbb{K}-\mathbb{K}_h) f_i\|_{0,E} + \|\mathbb{K} f_i - \Pi_k^{0,E} \mathbb{K} f_i\|_{0,E})\|\gamma \mathbb{K}^*_h f_i^*-\Pi_k^{0,E} (\gamma  \mathbb{K}^*_h f_i^*) \|_{0,E} \nonumber \\
& \qquad + \|\gamma\|_{0,\infty,E} (\|(\mathbb{K}-\mathbb{K}_h) f_i\|_{0,E} + \|\mathbb{K} f_i - \Pi_k^{0,E} \mathbb{K} f_i\|_{0,E}) \|\mathbb{K}^*_h f_i^*-\Pi_k^{0,E} \mathbb{K}^*_h f_i^*\|_{0,E} \} \nonumber \\
&\lesssim  \sum_{E \in \mathcal{T}_h} \{
(\|\gamma\|_{0,\infty,E}\|(\mathbb{K}-\mathbb{K}_h) f_i \|_{0,E} + h_E^{1+s}\|\gamma \mathbb{K} f_i \|_{1+s,E})\|\mathbb{K}^*_h f_i^*-\Pi_k^{0,E} \mathbb{K}^*_h f_i^* \|_{0,E} \nonumber \\
& \qquad + (\|(\mathbb{K}-\mathbb{K}_h) f_i\|_{0,E} +  h_E^{1+s} \|\mathbb{K} f_i \|_{1+s,E})\|\gamma \mathbb{K}^*_h f_i^*-\Pi_k^{0,E} (\gamma  \mathbb{K}^*_h f_i^*) \|_{0,E} \nonumber \\
& \qquad + \|\gamma\|_{0,\infty,E} (\|(\mathbb{K}-\mathbb{K}_h) f_i\|_{0,E} + h_E^{1+s} \|\mathbb{K} f_i \|_{1+s,E}) \|\mathbb{K}^*_h f_i^*-\Pi_k^{0,E} \mathbb{K}^*_h f_i^*\|_{0,E} \} \nonumber \\
&\lesssim h^{2(r+s)}. \label{a3.39}
\end{align}
Substituting \eqref{a3.37}, \eqref{a3.38} and \eqref{a3.39} into \eqref{a3.35}, we obtain
\begin{align}
|I\!I| \lesssim h^{2s}. \label{a3.40}
\end{align}
For the term $I\!I\!I$, by using \eqref{a3.3}, \eqref{a3.23} and \eqref{a3.22}, we have
\begin{align}
|I\!I\!I| &= \sum_{E \in \mathcal{T}_h} |(\mathbb{K}_h f_i-\Pi_k^{0,E} \mathbb{K}_h f_i,f_i^*)_{0,E}| \nonumber \\
&=\sum_{E \in \mathcal{T}_h} |(\mathbb{K}_h f_i-\Pi_k^{0,E} \mathbb{K}_h f_i,f_i^*-\Pi_k^{0,E} f_i^*)_{0,E}| \nonumber \\
& \leq \sum_{E \in \mathcal{T}_h} h_E^{1+s} \|f_i^*\|_{1+s,E} \|\mathbb{K}_h f_i-\Pi_k^{0,E} \mathbb{K} f_i\|_{0,E} \nonumber \\
& \leq \sum_{E \in \mathcal{T}_h} h_E^{1+s} \|f_i^*\|_{1+s,E} \|\mathbb{K}_h f_i -\mathbb{K} f_i + \mathbb{K} f_i -\Pi_k^{0,E} \mathbb{K} f_i\|_{0,E} \nonumber \\
& \leq \sum_{E \in \mathcal{T}_h} h_E^{1+s} \|f_i^*\|_{1+s,E} \{ \|\mathbb{K}_h f_i -\mathbb{K} f_i\|_{0,E} + \|\mathbb{K} f_i -\Pi_k^{0,E} \mathbb{K} f_i\|_{0,E} \}  \nonumber \\
& \leq h^{1+s} \|f_i^*\|_{1+s,\Omega} \{ \|(\mathbb{K}_h -\mathbb{K}) f_i\|_{0} + h^{1+s} \|\mathbb{K} f_i\|_{1+s,\Omega} \} \nonumber \\
& \lesssim h^{2(r+s)}. \label{a3.41}
\end{align}
Substituting \eqref{a3.33}, \eqref{a3.40} and \eqref{a3.41} into \eqref{a3.33}, and combining with \eqref{a3.32} and \eqref{a3.33}, we obtain \eqref{a3.27}.
\end{proof}

\section{Numerical experiments}
\indent In this section, we will present several numerical examples on polygonal meshes to verify our theoretical analysis. Our numerical experiments are implemented in MATLAB based on the software package developed in \cite{YuY2022}. 
When computing the high-order projections, we are inspired by the works in \cite{BeiraoBM2014, Beiraobmr2016}.
We select the unit square $\Omega=[0,1]^2$ as the computational domain, and adopt three different polygonal meshes for $\Omega$, as illustrated in Figure \ref{fig1}. To determine the parameter $\ell$ in \eqref{a3.6}, we refer to the setting in \cite{BerroneBF2025} and list the selected values of $\ell$ corresponding to $k=2, 3$ and $4$ in Table \ref{tab1}.

\begin{figure}[htbp]
  \centering
  \begin{subfigure}{0.3\textwidth}
    \centering
    \includegraphics[width=1.2\linewidth]{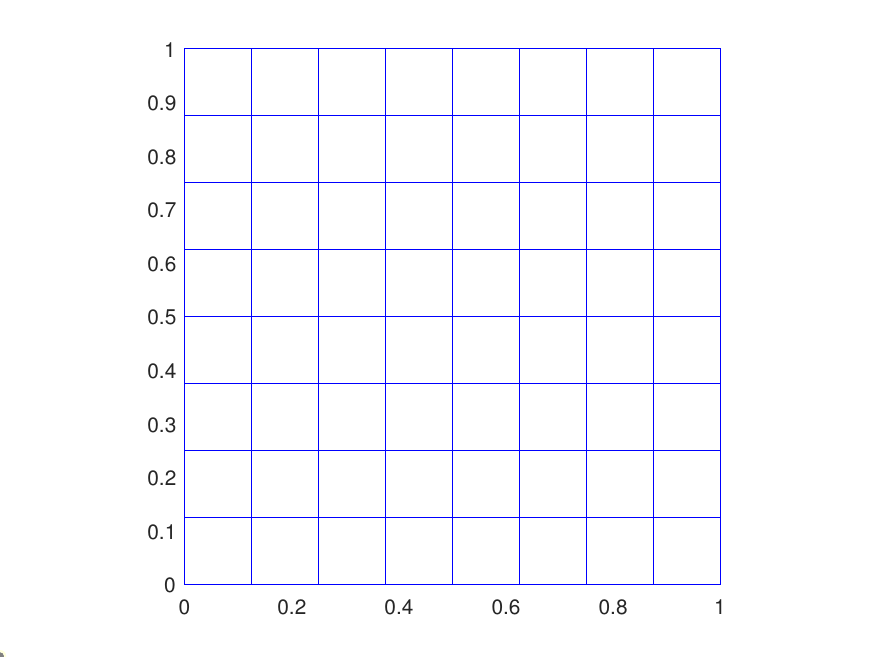}
    \subcaption*{$\mathcal{T}_h^1:$Convex quadrilateral mesh}
  \end{subfigure}
  \begin{subfigure}{0.3\textwidth}
    \centering
    \includegraphics[width=1.2\linewidth]{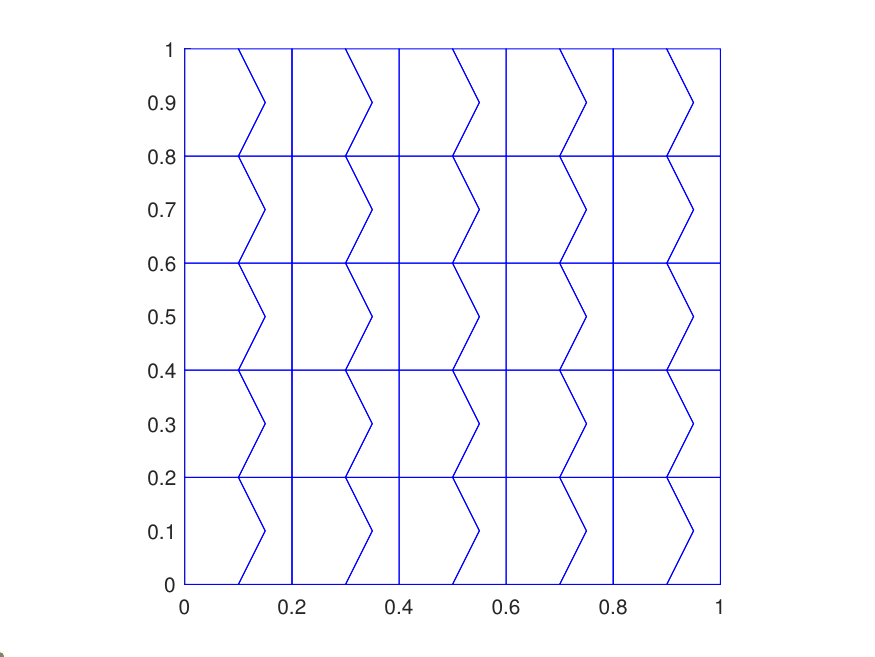}
     \caption*{$\mathcal{T}_h^2:$Pentagonal mesh}
  \end{subfigure}
  \begin{subfigure}{0.3\textwidth}
    \centering
    \includegraphics[width=1.2\linewidth]{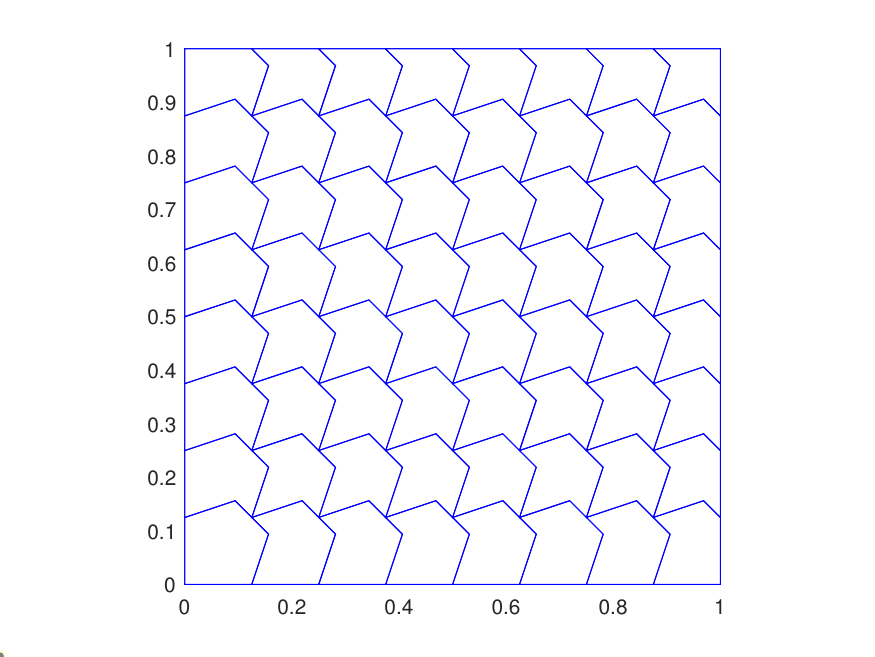}
    \caption*{$\mathcal{T}_h^3:$Concave octagon mesh}
  \end{subfigure}
  \caption{Different polygonal discretizations of the unit square.}
  \label{fig1}
\end{figure}

\begin{table}[h] \small
\setlength{\tabcolsep}{16pt}
  \begin{center}
  \caption{Values of $\ell$ corresponding to different values of $k$ and different mesh discretizations .}\label{tab1}
    \begin{tabular}{cccc}
\toprule
meshs & $k=2$ & $k=3$ & $k=4$ \\
\midrule
$\mathcal{T}_h^1$ &$\ell=1$ & $\ell=1$ & $\ell=1$ \\
$\mathcal{T}_h^2$ &$\ell=2$ & $\ell=2$ & $\ell=2$ \\
$\mathcal{T}_h^3$ &$\ell=3$ & $\ell=3$ & $\ell=3$ \\
\bottomrule
\end{tabular}
  \end{center}
\end{table}
Furthermore, in our numerical experiments, without loss of generality, we fix $\gamma \equiv 0$ and consider the following different values of the coefficients $\mathbf{\mathcal{K}}$ and $\boldsymbol{\beta}$.
\begin{itemize}
    \itemsep 0mm
    \item Case 1. $\mathbf{\mathcal{K}}=\begin{bmatrix} 1 & 0 \\ 0 & 1 \end{bmatrix}$, $\boldsymbol{\beta}= \begin{bmatrix} 1  \\ 0  \end{bmatrix}$;
    \item Case 2, $\mathbf{\mathcal{K}}=\begin{bmatrix} 1 & 0 \\ 0 & 1 \end{bmatrix}$, $\boldsymbol{\beta}= \begin{bmatrix} 10  \\ 0  \end{bmatrix}$;
    \item Case 3, $\mathbf{\mathcal{K}}=\begin{bmatrix} 8\times 10^{-3} & 0 \\ 0 & 1 \end{bmatrix}$, $\boldsymbol{\beta}= \begin{bmatrix} 0  \\ 0  \end{bmatrix}$.
\end{itemize}

\textbf{Example 1:} For Case 1, we compute the first five eigenvalues of \eqref{a3.13}
on meshes $\mathcal{T}_h^1$, $\mathcal{T}_h^2$ and $\mathcal{T}_h^3$ with $k=2$, $k=3$ and $k=4$, respectively. The exact eigenvalues are given by $\lambda_{1}=0.25 + 2\times\pi^2$, $\lambda_{2}=\lambda_{3}=0.25 + 5\times\pi^2$, $\lambda_{4}=0.25 + 8\times\pi^2$ and $\lambda_{5}=0.25 + 10\times\pi^2$. To characterize the convergence behavior, we plot the error curves in Figures \ref{fig2}, \ref{fig3} and \ref{fig4}.\\
\indent It can be observed from Figure \ref{fig2} and Figure \ref{fig3} that our stabilization-free virtual element method achieves the optimal convergence order for all the three meshes $\mathcal{T}_h^1$, $\mathcal{T}_h^2$ and $\mathcal{T}_h^3$.
In particular, when $k=2$, the fifth eigenvalue attains a better convergence order on mesh $\mathcal{T}_h^2$. However, for $k=4$, while the first five eigenvalues achieve the optimal convergence order on mesh $\mathcal{T}_h^2$, the convergence behavior of the first eigenvalue is unsatisfactory in the later stages on meshes $\mathcal{T}_h^1$ and $\mathcal{T}_h^3$.\\

\begin{figure}[htbp]
    \centering
    \setlength{\fboxsep}{0pt}
    \begin{minipage}[c]{0.32\linewidth}
        \centering
        \includegraphics[width=\linewidth]{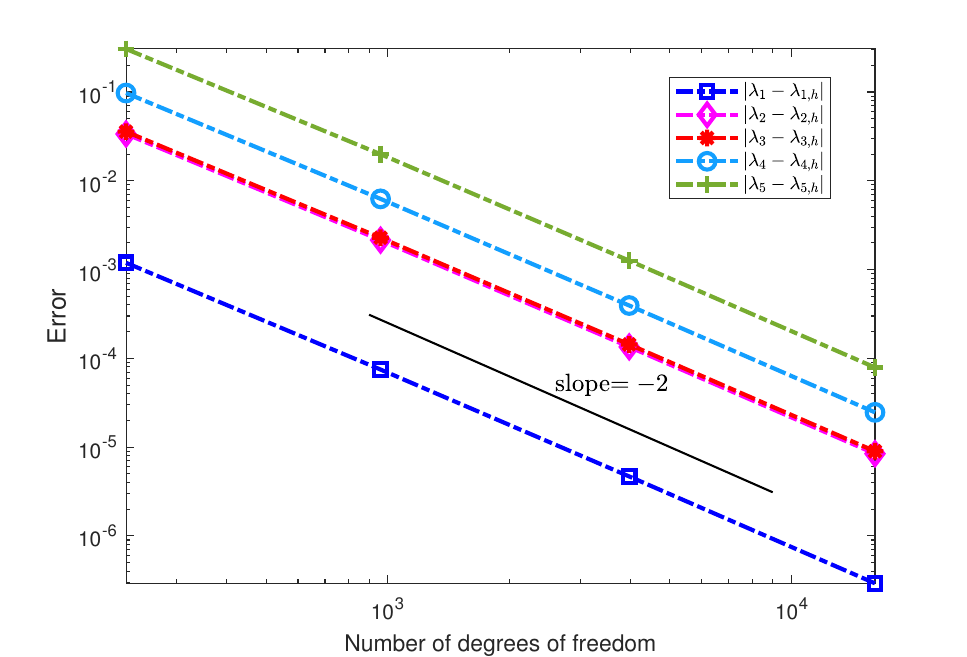}
    \end{minipage}
    \hfill{}
    \begin{minipage}[c]{0.32\linewidth}
        \centering
        \includegraphics[width=\linewidth]{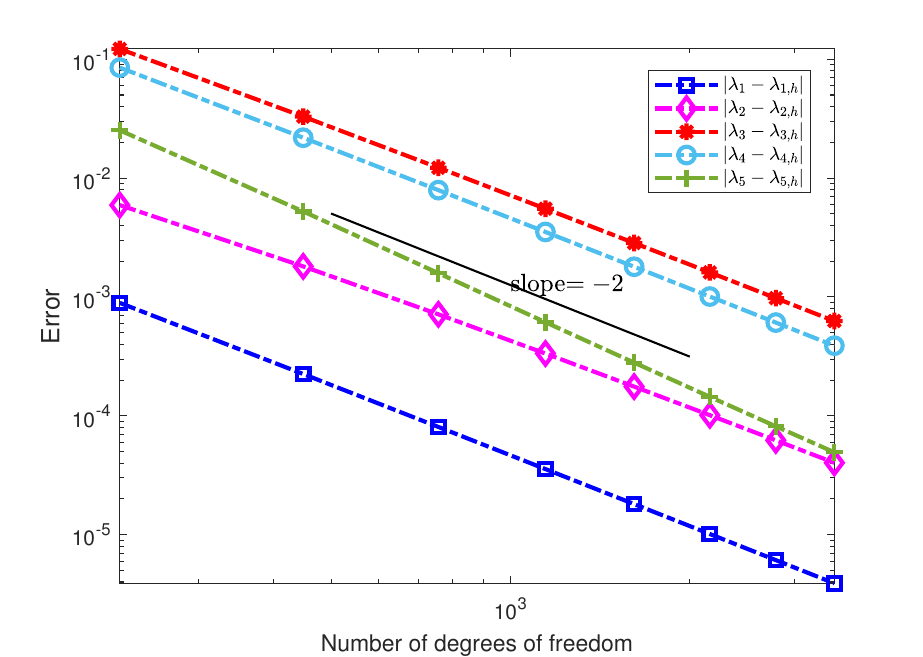}
    \end{minipage}%
    \hfill{}
    \begin{minipage}[c]{0.32\linewidth}
        \centering
        \includegraphics[width=\linewidth]{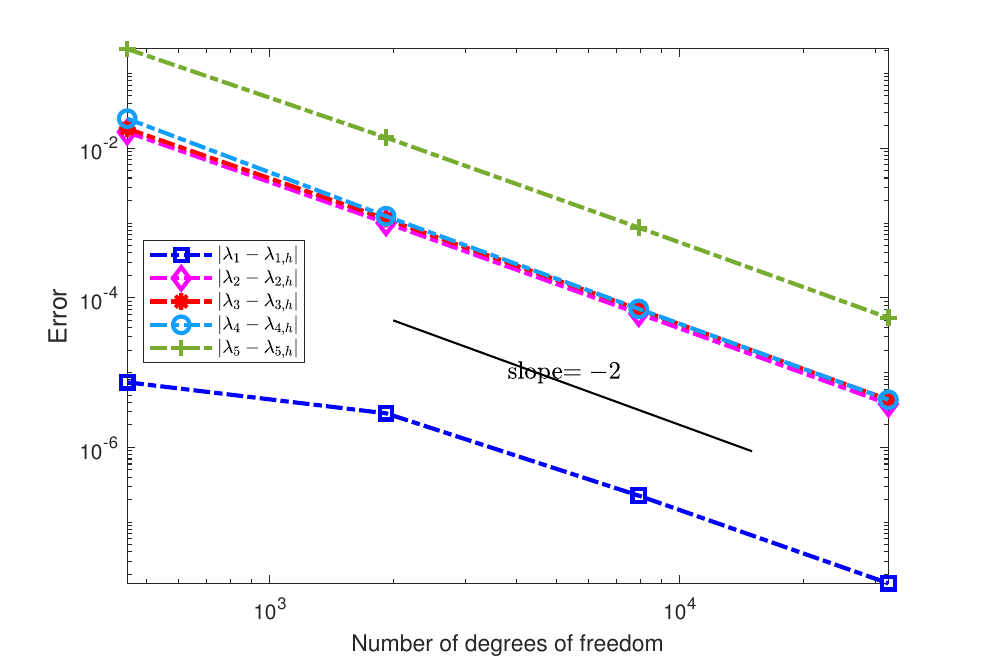}
    \end{minipage}
    \caption{In Case 1, the error curves for $k=2$: left ($\mathcal{T}_h^1$), middle ($\mathcal{T}_h^2$), right ($\mathcal{T}_h^3$).}
    \label{fig2}
\end{figure}

\begin{figure}[htbp]
    \centering
    \setlength{\fboxsep}{0pt}
    \begin{minipage}[c]{0.32\linewidth}
        \centering
        \includegraphics[width=\linewidth]{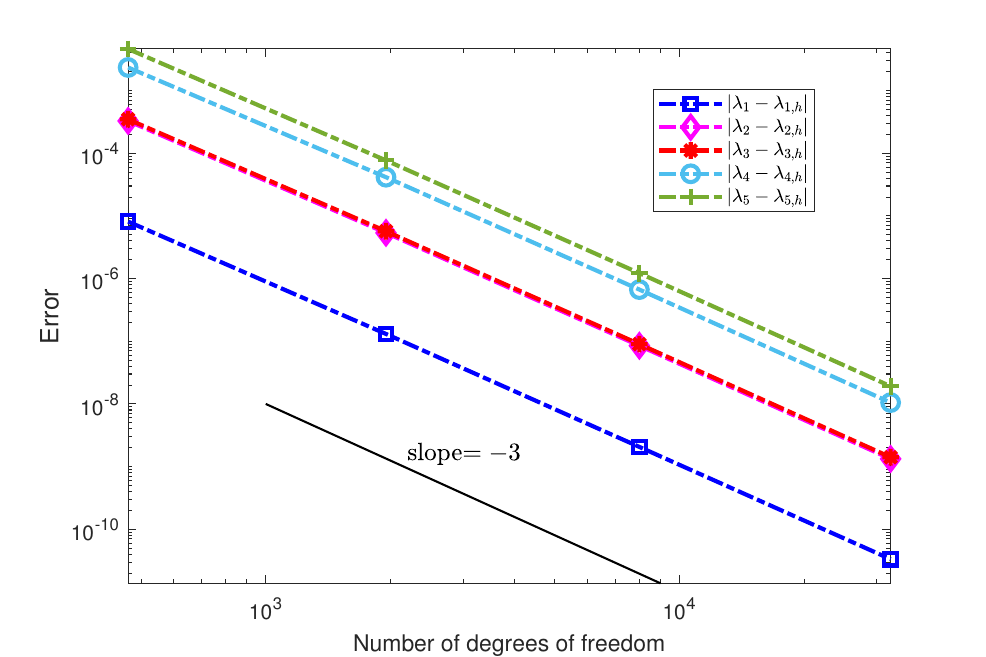}
    \end{minipage}
    \hfill{}
    \begin{minipage}[c]{0.32\linewidth}
        \centering
        \includegraphics[width=\linewidth]{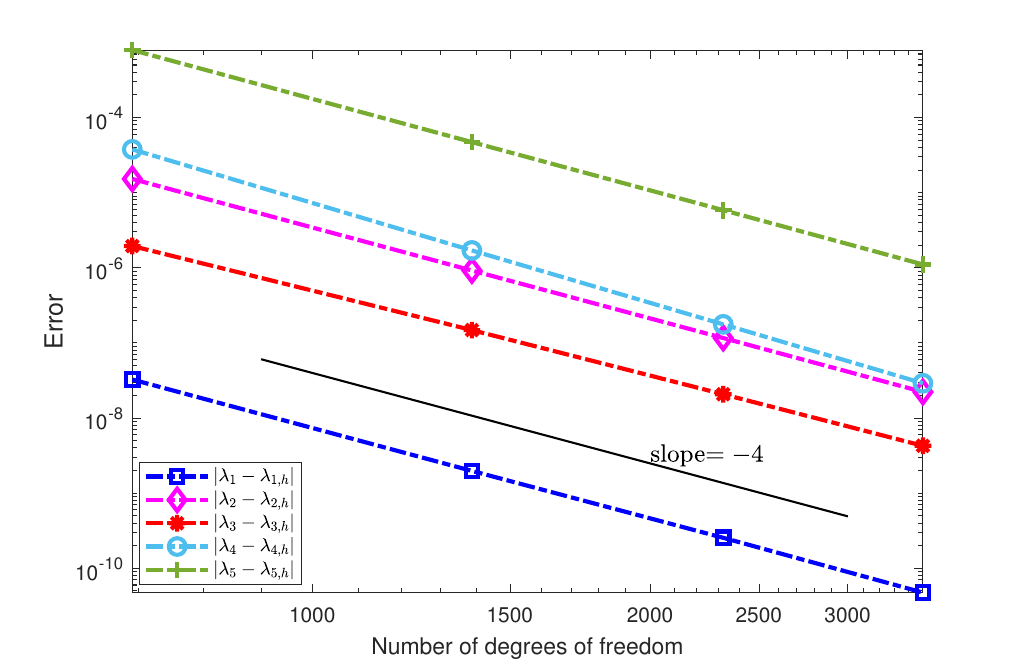}
    \end{minipage}%
    \hfill{}
    \begin{minipage}[c]{0.32\linewidth}
        \centering
        \includegraphics[width=\linewidth]{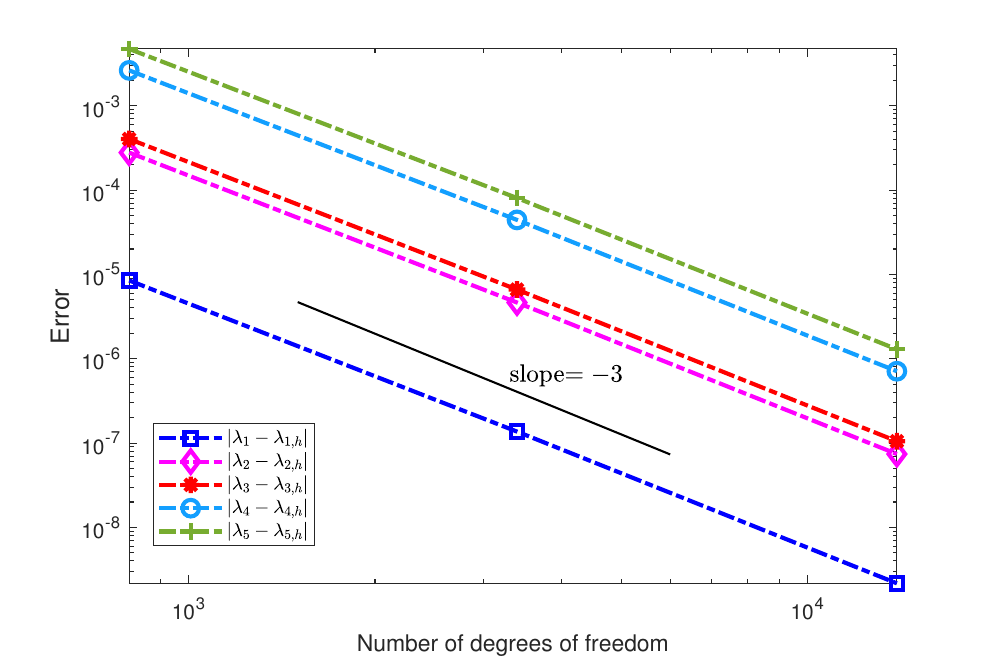}
    \end{minipage}
    \caption{In Case 1, the error curves for $k=3$: left ($\mathcal{T}_h^1$), middle ($\mathcal{T}_h^2$), right ($\mathcal{T}_h^3$).}
    \label{fig3}
\end{figure}

\begin{figure}[htbp]
    \centering
    \setlength{\fboxsep}{0pt}
    \begin{minipage}[c]{0.32\linewidth}
        \centering
        \includegraphics[width=\linewidth]{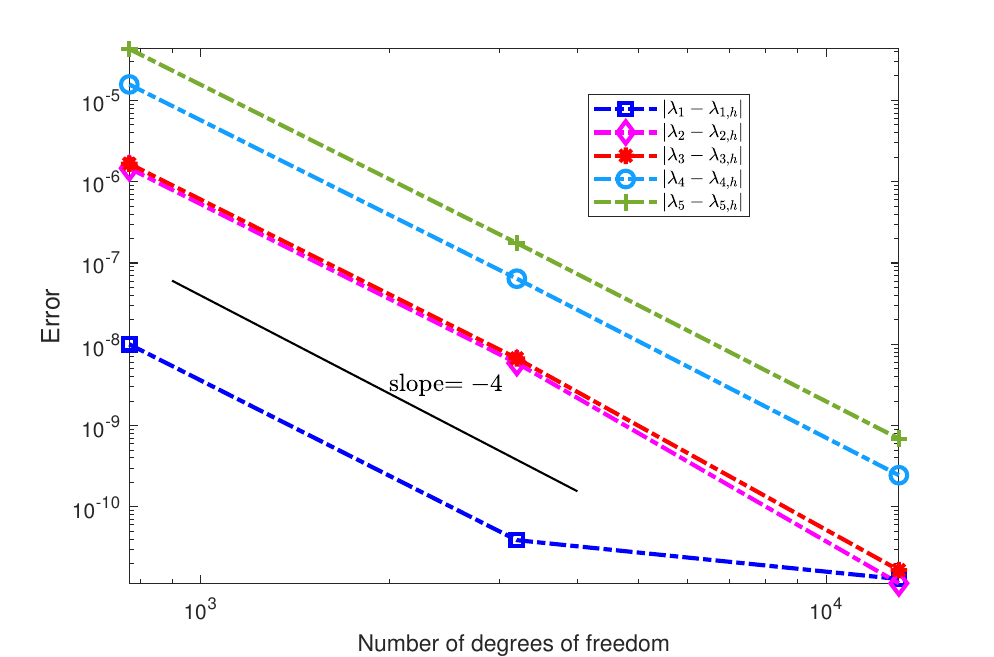}
    \end{minipage}
    \hfill{}
    \begin{minipage}[c]{0.32\linewidth}
        \centering
        \includegraphics[width=\linewidth]{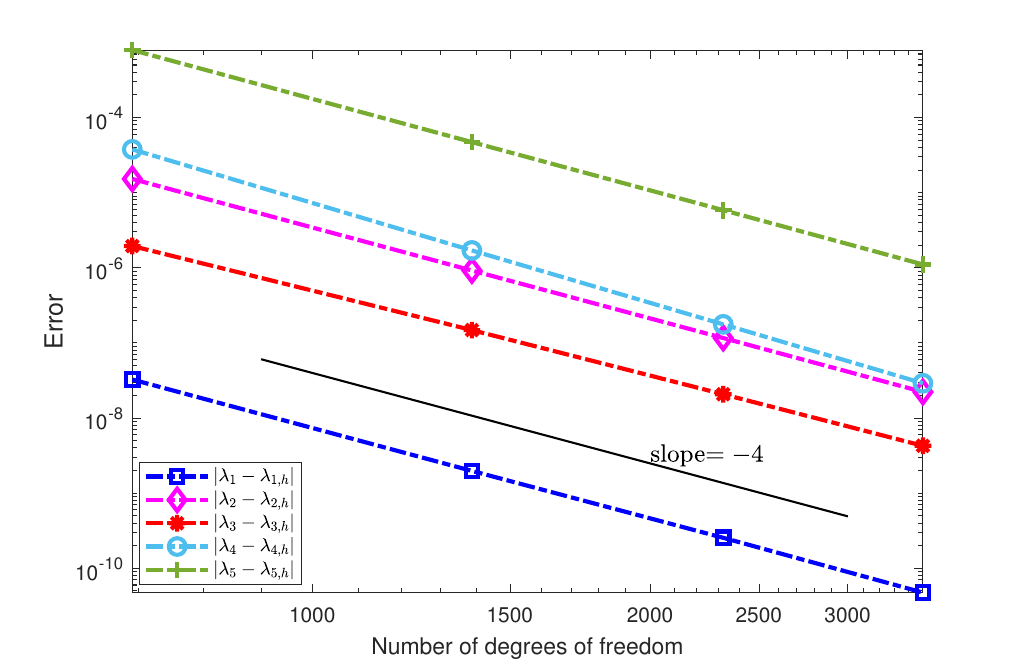}
    \end{minipage}%
    \hfill{}
    \begin{minipage}[c]{0.32\linewidth}
        \centering
        \includegraphics[width=\linewidth]{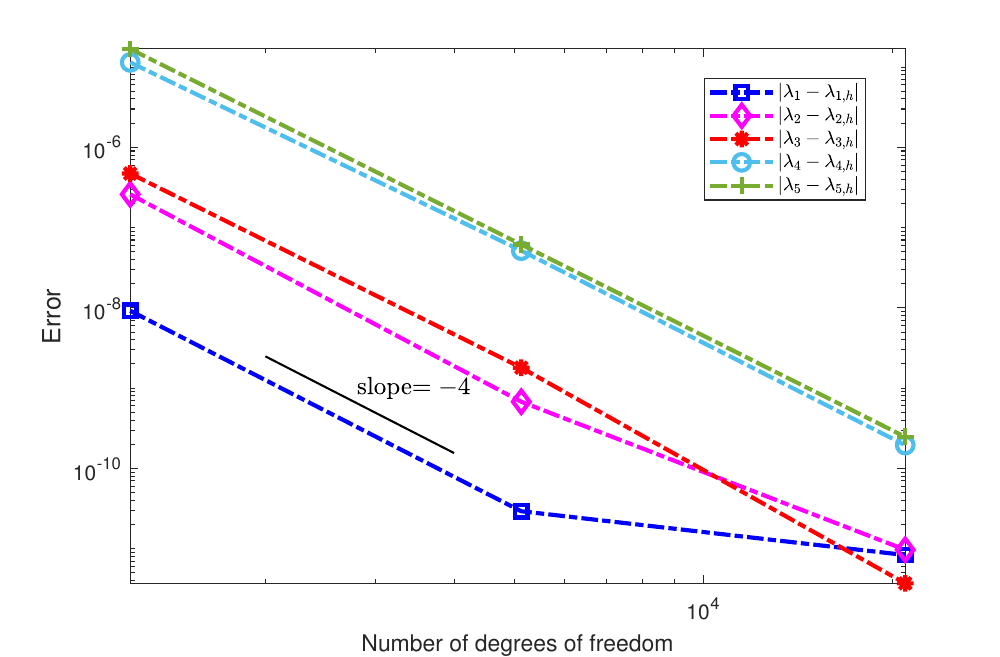}
    \end{minipage}
    \caption{In Case 1, the error curves for $k=4$: left ($\mathcal{T}_h^1$), middle ($\mathcal{T}_h^2$), right ($\mathcal{T}_h^3$).}
    \label{fig4}
\end{figure}

\textbf{Example 2:} For Case 2, we likewise compute the first five eigenvalues of \eqref{a3.13} with $k=2, 3, 4$ on meshes $\mathcal{T}_h^1$, $\mathcal{T}_h^2$, and $\mathcal{T}_h^3$. Here, the exact eigenvalues are $\lambda_{1}=25 + 2\pi^2$, $\lambda_{2}=\lambda_{3}=25 + 5\pi^2$, $\lambda_{4}=25 + 8\pi^2$, and $\lambda_{5}=25 + 10\pi^2$. We also plot the error curves in Figures \ref{fig5}, \ref{fig6}, and \ref{fig7}.\\

\indent As can be seen from Figure \ref{fig5}, when $k=2$, the third eigenvalue on meshes $\mathcal{T}_h^1$ and $\mathcal{T}_h^3$ exhibits somewhat unsatisfactory convergence behavior on coarse meshes; however, it ultimately achieves the optimal convergence order upon mesh refinement. Furthermore, in Figures \ref{fig5}-\ref{fig7}, the remaining cases all essentially attain the optimal convergence order.

\begin{figure}[htbp]
    \centering
    \setlength{\fboxsep}{0pt}
    \begin{minipage}[c]{0.32\linewidth}
        \centering
        \includegraphics[width=\linewidth]{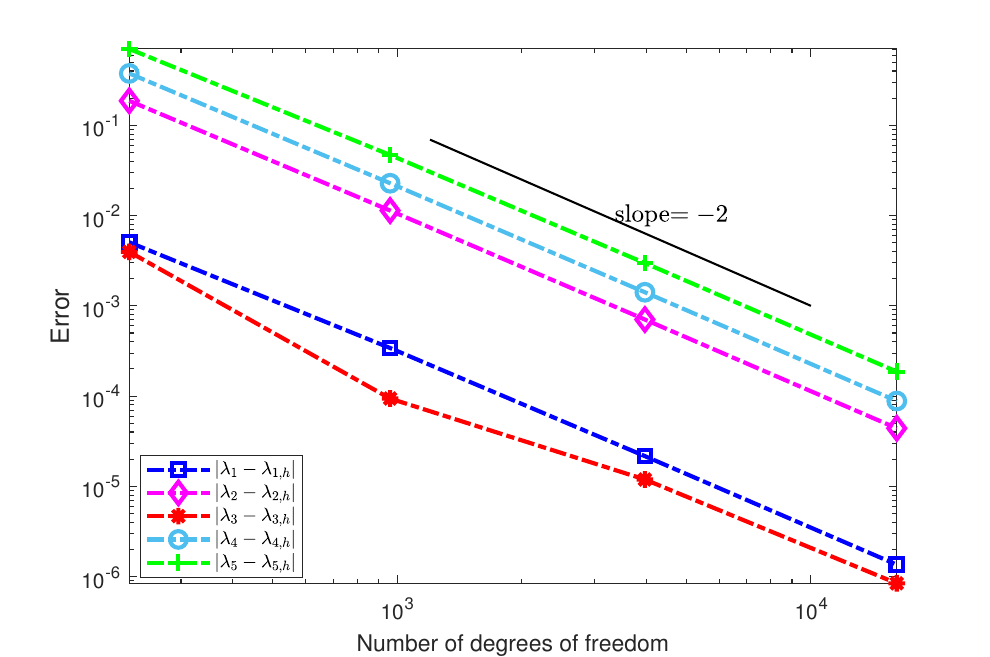}
    \end{minipage}
    \hfill{}
    \begin{minipage}[c]{0.32\linewidth}
        \centering
        \includegraphics[width=\linewidth]{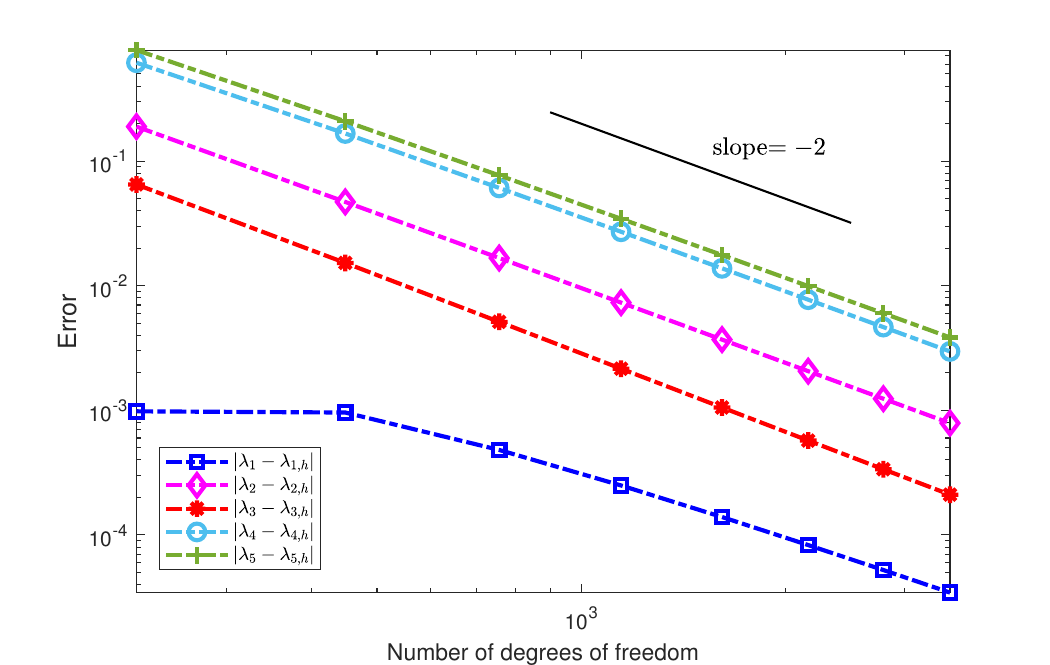}
    \end{minipage}%
    \hfill{}
    \begin{minipage}[c]{0.32\linewidth}
        \centering
        \includegraphics[width=\linewidth]{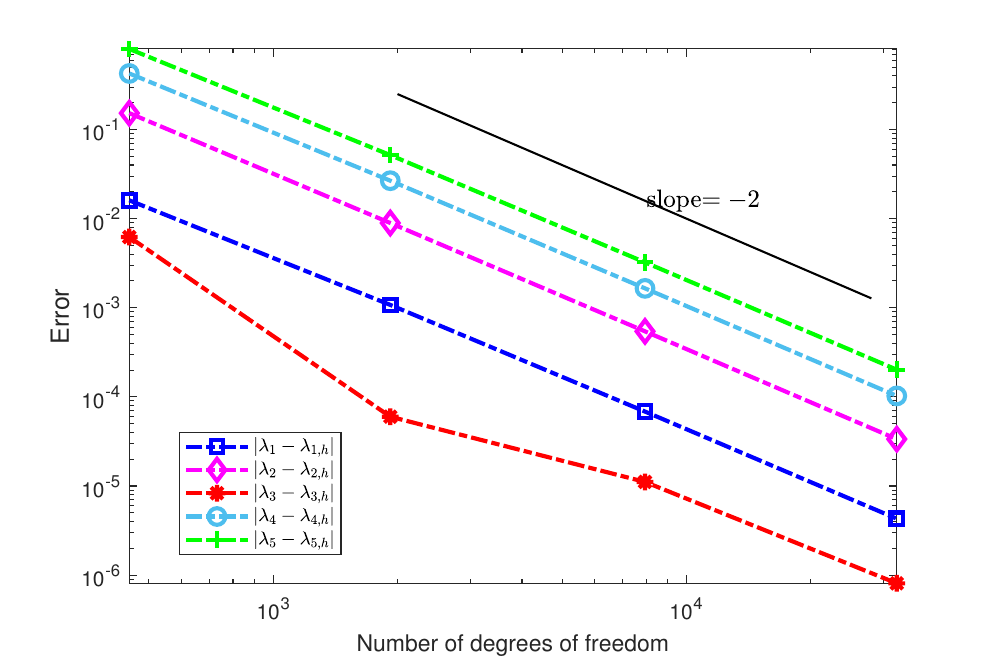}
    \end{minipage}
    \caption{In Case 2, the error curves for $k=2$: left ($\mathcal{T}_h^1$), middle ($\mathcal{T}_h^2$), right ($\mathcal{T}_h^3$).}
    \label{fig5}
\end{figure}

\begin{figure}[htbp]
    \centering
    \setlength{\fboxsep}{0pt}
    \begin{minipage}[c]{0.32\linewidth}
        \centering
        \includegraphics[width=\linewidth]{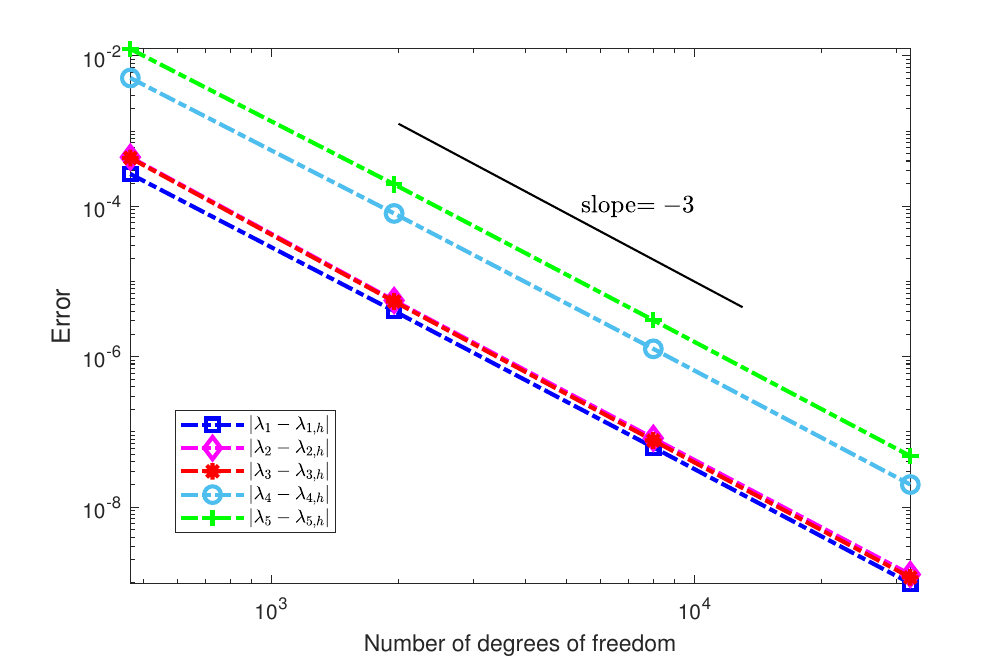}
    \end{minipage}
    \hfill{}
    \begin{minipage}[c]{0.32\linewidth}
        \centering
        \includegraphics[width=\linewidth]{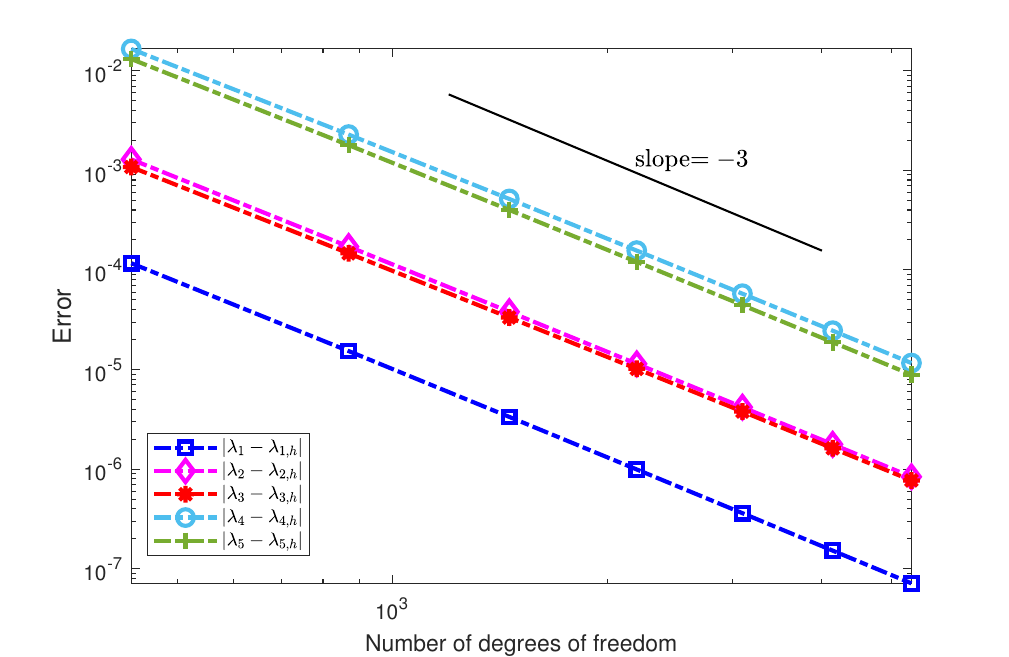}
    \end{minipage}%
    \hfill{}
    \begin{minipage}[c]{0.32\linewidth}
        \centering
        \includegraphics[width=\linewidth]{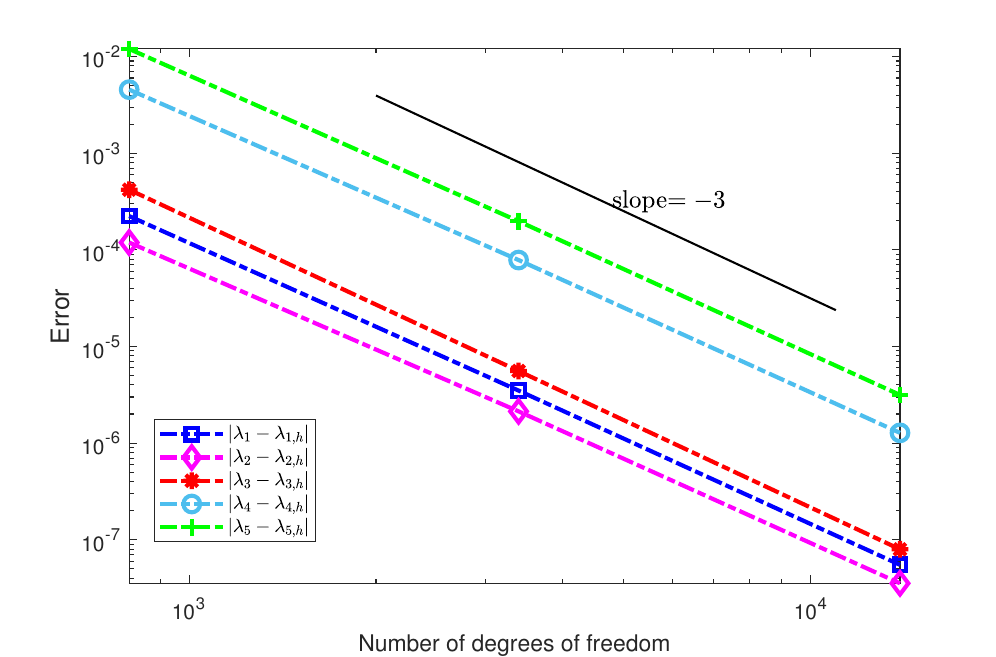}
    \end{minipage}
    \caption{In Case 2, the error curves for $k=3$: left ($\mathcal{T}_h^1$), middle ($\mathcal{T}_h^2$), right ($\mathcal{T}_h^3$).}
    \label{fig6}
\end{figure}

\begin{figure}[htbp]
    \centering
    \setlength{\fboxsep}{0pt}
    \begin{minipage}[c]{0.32\linewidth}
        \centering
        \includegraphics[width=\linewidth]{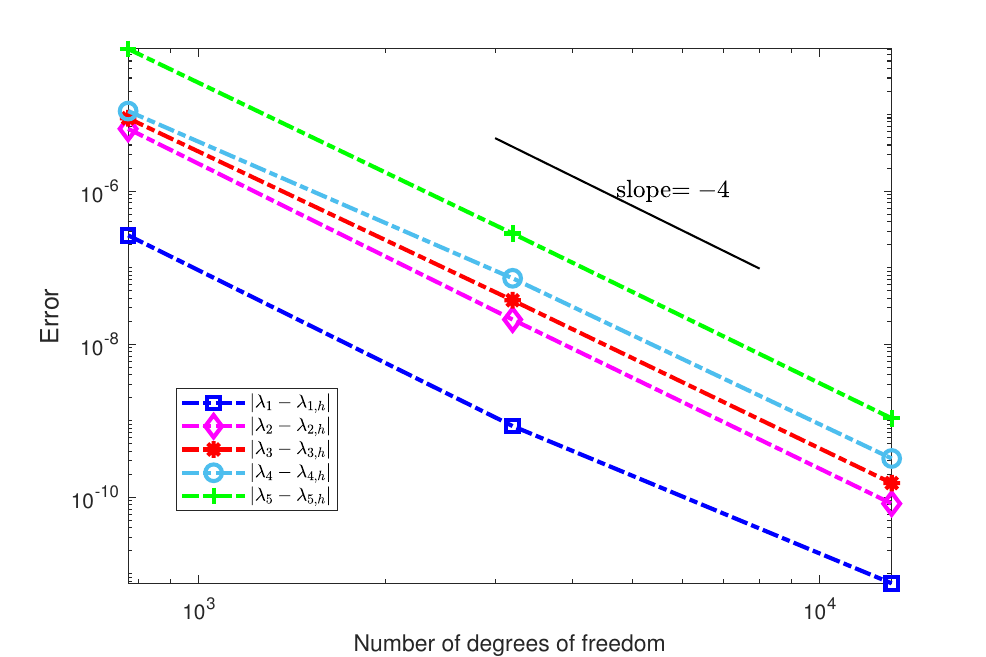}
    \end{minipage}
    \hfill{}
    \begin{minipage}[c]{0.32\linewidth}
        \centering
        \includegraphics[width=\linewidth]{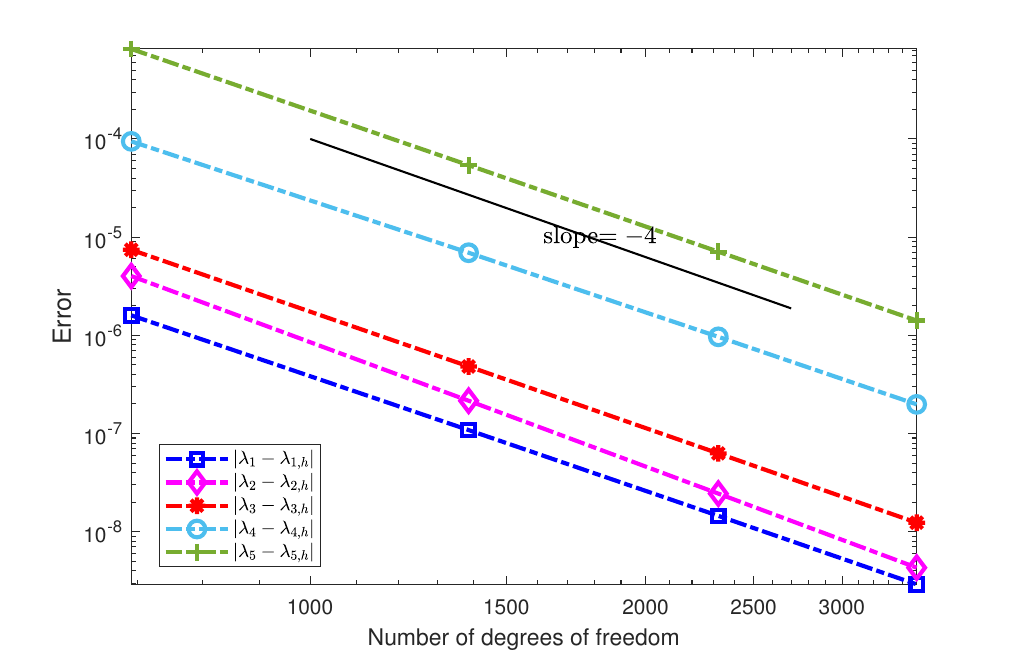}
    \end{minipage}%
    \hfill{}
    \begin{minipage}[c]{0.32\linewidth}
        \centering
        \includegraphics[width=\linewidth]{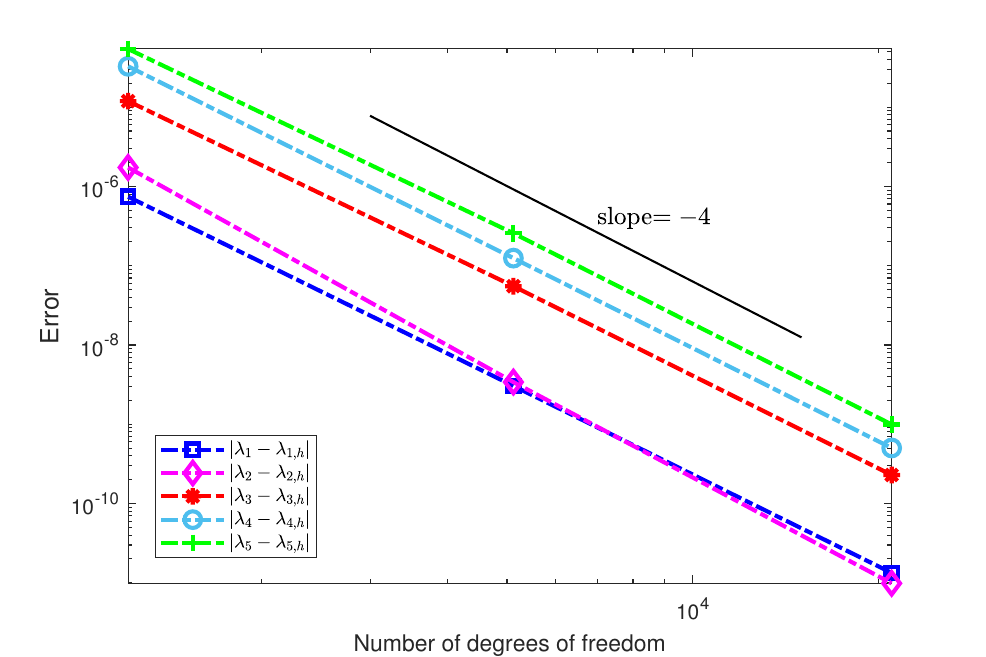}
    \end{minipage}
    \caption{In Case 2, the error curves for $k=4$: left ($\mathcal{T}_h^1$), middle ($\mathcal{T}_h^2$), right ($\mathcal{T}_h^3$).}
    \label{fig7}
\end{figure}

\textbf{Example 3:} In this example, we will compare the performance of the stabilization-free virtual element method (SFVEM) with the standard virtual element method (SVEM).
The SVEM discretization for the problem \eqref{a2.2} is to find $(\lambda_h, u_h) \in \mathbb{C} \times \mathcal{V}_{h,k}$ such that
\begin{eqnarray}
\sum_{E \in \mathcal{T}_h} \hat{a}_h^E(u_h,v_h)+b_h^E(u_h,v_h)+c_h^E(u_h,v_h) = \lambda_h \sum_{E \in \mathcal{T}_h}  \left(\Pi_k^{0,E} u_h,\Pi_k^{0,E} v_h\right)_{0,E},~ \forall v_h \in~\mathcal{V}_{h,k}, \label{a4.1}
\end{eqnarray}
where $\hat{a}_h^E(u_h,v_h)=\left(\mathcal{K}\Pi_{k-1}^{0,E}\nabla u_h,\Pi_{k-1}^{0,E}\nabla v_h\right)_E + \|\mathcal{K}\|_{0,\infty,E}\, s_a^E\!\left((I-\Pi_k^{\nabla,E})u_h,(I-\Pi_k^{\nabla,E})v_h\right),$ and $s_a^E(\cdot,\cdot)$ is a local symmetric dofi-dofi type stabilizing bilinear form. Subsequently, for Case 3, we compute the first eigenvalues of both \eqref{a3.13} and \eqref{a4.1} with $k=2$ on meshes $\mathcal{T}_h^1$, $\mathcal{T}_h^2$, and $\mathcal{T}_h^3$, respectively.
We use the exact eigenvalue $\lambda_1=(1+8 \times 10^{-3})\pi^2$ as the reference value to plot the error curves in Figure \ref{fig8} to illustrate the convergence behavior of the approximate eigenvalues.

\begin{figure}[htbp]
    \centering
    \setlength{\fboxsep}{0pt}
    \begin{minipage}[c]{0.32\linewidth}
        \centering
        \includegraphics[width=\linewidth]{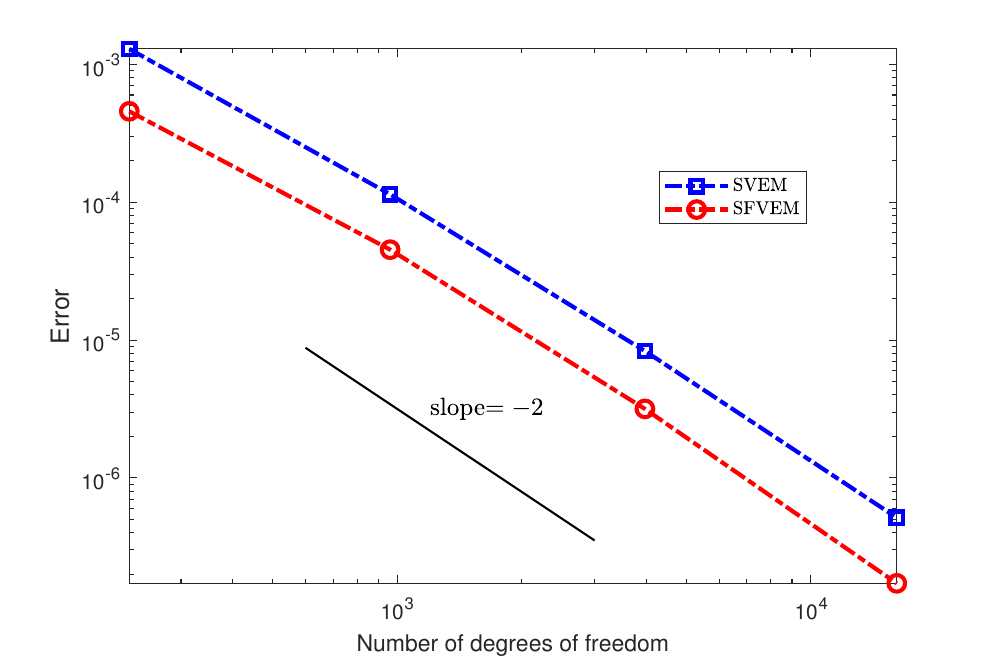}
    \end{minipage}
    \hfill{}
    \begin{minipage}[c]{0.32\linewidth}
        \centering
        \includegraphics[width=\linewidth]{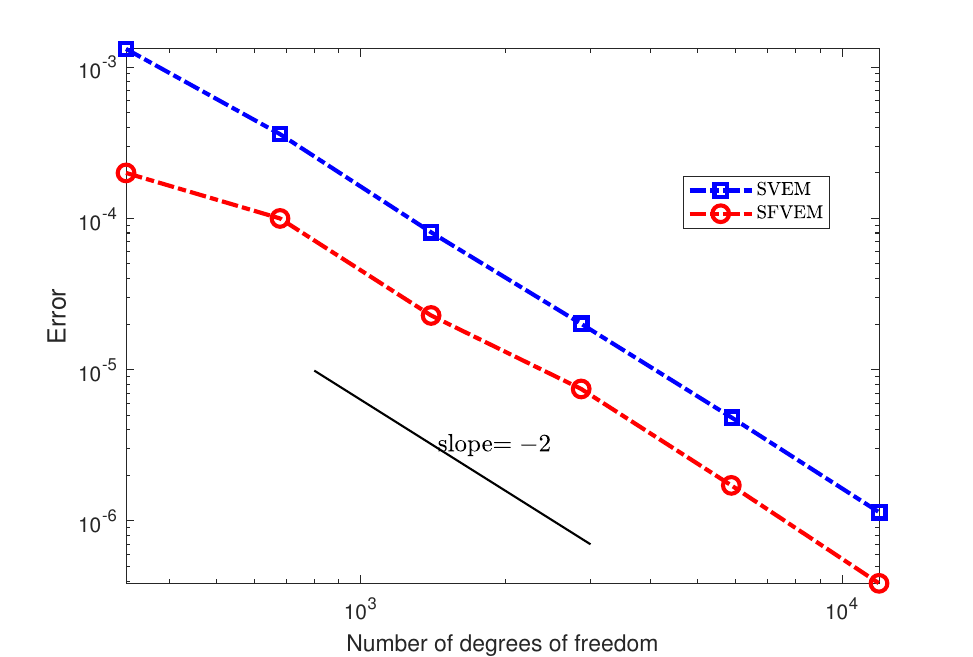}
    \end{minipage}%
    \hfill{}
    \begin{minipage}[c]{0.32\linewidth}
        \centering
        \includegraphics[width=\linewidth]{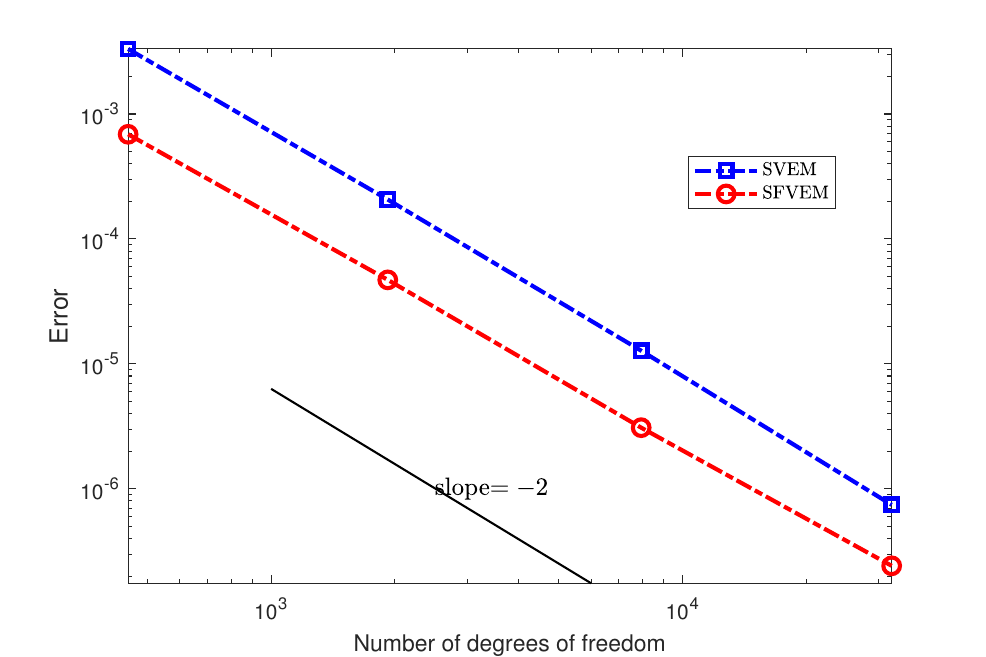}
    \end{minipage}
    \caption{Error curves for $k=2$: left ($\mathcal{T}_h^1$), middle ($\mathcal{T}_h^2$), right ($\mathcal{T}_h^3$).}
    \label{fig8}
\end{figure}
\indent From Figure \ref{fig8}, it is observed that both the SFVEM and the SVEM achieve the optimal convergence order on meshes $\mathcal{T}_h^1$, $\mathcal{T}_h^2$, and $\mathcal{T}_h^3$.
Furthermore, a clear comparison reveals that our SFVEM yields smaller errors than the SVEM under the same degree of freedom.

\section*{Author contributions}
Liangkun Xu: Conceptualization, Software, Methodology, Validation, Data curation, Writing—original draft preparation;
Shixi Wang: Conceptualization, Software, Methodology, Validation, Formal analysis, Writing—original draft preparation;
Hai Bi: Methodology, Validation, Formal analysis, Writing—review and editing, Supervision, Funding acquisition;
Yidu Yang: Methodology, Validation, Formal analysis, Writing—review and editing, Supervision.
All authors have read and agreed to the published version of the manuscript.

\section*{Acknowledgments}
This work was supported by the National Natural Science Foundation of China (Grant No. 12261024).

\section*{Conflict of interest}
The authors declare that they have no conflict of interest.


\begin{thebibliography}{999}
\bibitem{BeiraoBC2013}
L. Beir$\tilde{a}$o da Veiga, F. Brezzi, A. Cangiani, G. Manzini, L. D. Marini, A. Russo,  {Basic principles of virtual element methods},
\newblock  {\it Math. Models Methods Appl. Sci.,}
\textbf{23} (2013), 199-214. \url{http://dx.doi.org/10.1142/s0218202512500492}

\bibitem{BeiraoBM2014}
L. Beir$\tilde{a}$o da Veiga, F. Brezzi,  L. D. Marini, A. Russo,  {The hitchhiker's guide to the virtual element method},
\newblock  {\it Math. Models Methods Appl. Sci.,}
\textbf{24} (2014), 1541-1573. \url{http://dx.doi.org/10.1142/S021820251440003X}

\bibitem{BrezziFM2014}
F. Brezzi, R. S. Falk, L. D. Marini, {Basic principles of mixed virtual element methods},
\newblock  {\it ESAIM: Math. Model. Numer. Anal.},
 \textbf{48} (2014), 1227-1240. \url{http://dx.doi.org/10.1051/m2an/2013138}


\bibitem{DeDiosLM2016}
B. A. De Dios, K. Lipnikov, G. Manzini,  {The nonconforming virtual element method},
\newblock  {\it ESAIM: Math. Model. Numer. Anal.},
\textbf{50} (2016), 879-904. \url{http://dx.doi.org/10.1051/m2an/2015090}


\bibitem{BeiraoBreMar2016}
L. Beir$\tilde{a}$o da Veiga, F. Brezzi, L. D. Marini, A. Russo, {Serendipity nodal VEM spaces},
\newblock  {\it Comput. Fluids},
\textbf{141} (2016), 2-12. \url{http://dx.doi.org/10.1016/j.compfluid.2016.02.015}

\bibitem{MascottoPP2018}
L. Mascotto, I. Perugia, A. Pichler, {Non-conforming harmonic virtual element method: h-and p-versions},
\newblock  {\it J. Sci. Comput.}, \textbf{77} (2018), 1874-1908. \url{https://doi.org/10.1007/s10915-018-0797-4}

\bibitem{BenvenutiCM2019}
E. Benvenuti, A. Chiozzi, G. Manzini, N. Sukumar, {Extended virtual element method for the Laplace problem with singularities and discontinuities},
\newblock  {\it Comput. Methods Appl. Mech. Eng.}, \textbf{356} (2019), 571-597. \url{https://doi.org/10.1016/j.cma.2019.07.028}

\bibitem{CaoCG2022}
S. Cao, L. Chen, R. Guo, F. Lin,  {Immersed virtual element methods for elliptic interface problems in two dimensions},
 \newblock  {\it J. Sci. Comput.}, \textbf{93} (2022), 12. \url{https://doi.org/10.1007/s10915-022-01949-x}


\bibitem{ZhaoZ2023}
J. Zhao, S. Mao, B. Zhang, F. Wang, {The interior penalty virtual element method for the biharmonic problem},
 \newblock  {\it  Math. Comput.}, \textbf{92} (2023), 1543-1574. \url{https://doi.org/10.1090/mcom/3828}


\bibitem{ChenWZ2023}
F. Chen, Q. Wang, Z. Zhou, {Two-grid virtual element discretization of semilinear elliptic problem},
 \newblock  {\it Appl. Numer. Math.}, \textbf{186} (2023), 228-240. \url{https://doi.org/10.1016/j.apnum.2023.01.009}



\bibitem{BoffiGG2020}
D. Boffi, F. Gardini, L. Gastaldi, {Approximation of PDE eigenvalue problems involving parameter dependent matrices},
\newblock  {\it Calcolo}, \textbf{57} (2020), 41. \url{https://doi.org/10.1007/s10092-020-00390-6}

\bibitem{BerroneBM2022}
S. Berrone, A. Borio, F. Marcon, {Comparison of standard and stabilization free virtual elements on anisotropic elliptic problems},
\newblock  {\it Appl. Math. Lett.}, \textbf{129} (2022), 107971. \url{https://doi.org/10.1016/j.aml.2022.107971}

\bibitem{BeiraoCNRH2023}
L. Beir$\tilde{a}$o da Veiga, C. Canuto, R. H. Nochetto, M. Verani, {Adaptive VEM: stabilization-free a posteriori error analysis and contraction property},
\newblock  {\it SIAM J. Numer. Anal.}, \textbf{61} (2023), 457-494. \url{https://doi.org/10.1137/21M1458740}


\bibitem{Mascotto2023}
 L. Mascotto, {The role of stabilization in the virtual element method: a survey},
\newblock  {\it Comput. Math. Appl.},
\textbf{151} (2023), 244-251. \url{https://doi.org/10.1016/j.camwa.2023.09.045}

\bibitem{AlzabenBD2025}
L. Alzaben, D. Boffi, A. Dedner, L. Gastaldi, {On the stabilization of a virtual element method for an acoustic vibration problem},
\newblock  {\it Math. Models Methods Appl. Sci.},
 \textbf{35} (2025), 655-701. \url{https://doi.org/10.1142/s0218202525500071}

\bibitem{BerroneBOFA2026}
S. Berrone, A. Borio, D. Fassino, et al. {A residual a posteriori error estimate for the stabilization-free virtual element method}.
\newblock  {\it J. Comput. Phys.}, \textbf{553} (2026), 114704.  \url{https://doi.org/10.1016/j.jcp.2026.114704}

\bibitem{BerroneBOMA2023}
S. Berrone, A. Borio, F. Marcon, G. Teora, {A first-order stabilization-free virtual element method},
\newblock  {\it Appl. Math. Lett.},
\textbf{142} (2023), 108641. \url{https://doi.org/10.1016/j.aml.2023.108641}

\bibitem{BerroneBF2025}
S. Berrone, A. Borio, D. Fassino, F. Marcon,  {Stabilization-free Virtual Element Method for 2D second order elliptic equations},
\newblock  {\it Comput. Methods Appl. Mech. Eng.},
\textbf{438} (2025), 117839. \url{https://doi.org/10.1016/j.cma.2025.117839}

\bibitem{BerroneBM2025}
S. Berrone, A. Borio, F. Marcon, {Lowest order stabilization free virtual element method for the 2D Poisson equation},
\newblock  {\it Comput. Math. Appl.},
\textbf{177} (2025), 78-99. \url{https://doi.org/10.1016/j.camwa.2024.11.017}

\bibitem{DAltriAMdp2021}
A. M. D'Altri, S. de Miranda, L. Patruno, E. Sacco, {An enhanced VEM formulation for plane elasticity},
\newblock  {\it Comput. Methods Appl. Mech. Eng.},
 \textbf{376} (2021), 113663.  \url{https://doi.org/10.1016/j.cma.2020.113663}


\bibitem{ChenSukumar2023}
A. Chen, N. Sukumar, {Stabilization-free serendipity virtual element method for plane elasticity},
\newblock  {\it Comput. Methods Appl. Mech. Eng.},
 \textbf{404} (2023), 115784. \url{https://doi.org/10.1016/j.cma.2022.115784}

\bibitem{BertrandCG2023}
F. Bertrand, C. Carstensen, B. Gr\"{a}{\ss}le, N. T. Tran, {Stabilization-free HHO a posteriori error control},
\newblock  {\it Numer. Math.},
\textbf{154} (2023), 369-408.  \url{https://doi.org/10.1007/s00211-023-01366-8}

\bibitem{XuPeng2023}
B. B. Xu, F. Peng, P. Wriggers,  {Stabilization-free virtual element method for finite strain applications},
\newblock  {\it Comput. Methods Appl. Mech. Eng.},
 \textbf{417} (2023), 116555.  \url{https://doi.org/10.1016/j.cma.2023.116555}

\bibitem{LampertiCP2023}
A. Lamperti, M. Cremonesi, U. Perego,  A. Russo, C. Lovadina, {A Hu-Washizu variational approach to self-stabilized virtual elements: 2D linear elastostatics},
\newblock  {\it Comput. Mech.},
\textbf{71} (2023), 935-955. \url{https://doi.org/10.1007/s00466-023-02282-2}

\bibitem{BouchezGB2024}
T. Bouchez, A. Gravouil, N. Blal, A. Giacoma, E. Delor, J. D. Beley,  {A Hu-Washizu stabilization-free Virtual Element Method for 3D linear elasticity with star-convex polyhedrons},
\newblock  {\it Comput. Methods Appl. Mech. Eng.},
\textbf{432} (2024), 117420. \url{https://doi.org/10.1016/j.cma.2024.117420}

\bibitem{BorioLM2024}
A. Borio, C. Lovadina, F. Marcon, M. Visinoni, {A lowest order stabilization-free mixed virtual element method},
\newblock  {\it Comput. Math. Appl.},
\textbf{160} (2024), 161-170.  \url{https://doi.org/10.1016/j.camwa.2024.02.024}

\bibitem{BerroneBM2024}
S. Berrone, A. Borio, F. Marcon, {A stabilization-free virtual element method based on divergence-free projections},
\newblock  {\it Comput. Methods Appl. Mech. Eng.},
\textbf{424} (2024), 116885. \url{https://doi.org/10.1016/j.cma.2024.116885}

\bibitem{MengWB2022}
J. Meng, X. Wang, L. L. Bu, L. Mei, {A lowest-order free-stabilization virtual element method for the Laplacian eigenvalue problem},
\newblock  {\it J. Comput. Appl. Math.},
\textbf{410} (2022), 114013. \url{https://doi.org/10.1016/j.cam.2021.114013}

\bibitem{MarconMora2025}
F. Marcon, D. Mora,  {A Stabilization-Free Virtual Element Method for the Convection-Diffusion Eigenproblem},
\newblock {\it J. Sci. Comput.},
\textbf{102} (2025), 46. \url{https://doi.org/10.1007/s10915-024-02765-1}

\bibitem{MengGQM2026}
J. Meng, L. Guan, X. Qian, S. Song, L. Mei,  {Stabilization-Free Virtual Element Method for the Transmission Eigenvalue Problem on Anisotropic Media},
\newblock  {\it J. Comput. Math.},
\textbf{44} (2026), 103-134. \url{https://doi.org/10.4208/jcm.2410-m2024-0023}

\bibitem{FolignoBCV2026}
P. P. Foligno, D. Boffi, F. Credali, R. Vescovini, {Benchmarking stabilized and self-stabilized p-virtual element methods with variable coefficients}
\newblock  {\it Comput. Methods Appl. Mech. Eng.},
 \textbf{455} (2026), 118863. \url{https://doi.org/10.1016/j.cma.2026.118863}

\bibitem{CangianiMS2017}
A. Cangiani, G. Manzini, O. J. Sutton, {Conforming and nonconforming virtual element methods for elliptic problems},
\newblock  {\it IMA J. Numer. Anal.},
 \textbf{37} (2017), 1317-1354. \url{https://doi.org/10.1093/imanum/drw036}

\bibitem{BeiraoBM2016}
L. Beir$\tilde{a}$o da Veiga, F. Brezzi, L. D. Marini, A. Russo, {Virtual element method for general second-order elliptic problems on polygonal meshes},
\newblock  {\it Math. Models Methods Appl. Sci.},
\textbf{26} (2016), 729-750. \url{https://doi.org/10.1142/S0218202516500160}

\bibitem{BernardiV2000}
C. Bernardi, R. Verf\"{u}rth, {Adaptive finite element methods for elliptic equations with non-smooth coefficients},
\newblock  {\it Numer. Math.},
\textbf{85} (2000), 579-608. \url{https://doi.org/10.1007/PL00005393}

\bibitem{Grisvard2011}
\newblock P. Grisvard,  {\em Elliptic problems in nonsmooth domains},
\newblock  Society for Industrial and Applied Mathematics, 2011. \url{https://epubs.siam.org/doi/book/10.1137/1.9781611972030}

\bibitem{AhmadAB2013}
B. Ahmad, A. Alsaedi, F. Brezzi,  L. D. Marini, A. Russo, {Equivalent projectors for virtual element methods},
\newblock  {\it Comput. Math. Appl.},
\textbf{66} (2013), 376-391. \url{https://doi.org/10.1016/j.camwa.2013.05.015}

\bibitem{BaO1991}
I. Babu$\check{s}$ka  J. E. Osborn,
\newblock  {\em Eigenvalue Problems, Handbook of Numerical Analysis, Vol.II, Finite Element Methods (Part 1)},
\newblock  Edited by P. G. Ciarlet and J. L. Lions, North-Holland, Elsevier Science Publishers B.V, 1991.
\url{https://doi.org/10.1016/S1570-8659(05)80042-0}

\bibitem{MoraRiveraRodriguez2015}
D. Mora, G. Rivera, R. Rodr\'{\i}guez, {A virtual element method for the Steklov eigenvalue problem},
\newblock  {\it Math. Models Methods Appl. Sci.},
\textbf{25} (2015), 1421-1445. \url{https://doi.org/10.1142/S0218202515500372}

\bibitem{YuY2022}
Y. Yu, {mVEM: a MATLAB software package for the virtual element methods},
\newblock arXiv preprint, 2022, 2204.01339. \url{https://doi.org/10.48550/arXiv.2204.01339}

\bibitem{Beiraobmr2016}
L. Beir$\tilde{a}$o da Veiga, F. Brezzi, L.D. Marini, A. Russo, {Virtual Element Implementation for General Elliptic Equations. In: G. Barrenechea, F. Brezzi, A. Cangiani, E. Georgoulis, (eds) Building Bridges: Connections and Challenges in Modern Approaches to Numerical Partial Differential Equations. Lecture Notes in Computational Science and Engineering},
 \textbf{114} (2016), Springer, Cham.
\url{https://doi.org/10.1007/978-3-319-41640-3_2}


\end{thebibliography}
\end{document}